\documentclass[reqno, 10pt]{amsart}

\usepackage{amssymb,  latexsym}
\usepackage{hyperref}
\usepackage{amsmath}
\usepackage{enumerate}
\usepackage{amsfonts}
\usepackage{graphicx}
\usepackage{mathrsfs}
\usepackage{color}

\theoremstyle{plain}
\newtheorem{theorem}{Theorem}
\newtheorem{proposition}[theorem]{Proposition}
\newtheorem{lemma}[theorem]{Lemma}
\newtheorem{corollary}[theorem]{Corollary}

\theoremstyle{definition}

\newtheorem{remark}[theorem]{Remark}

\numberwithin{equation}{section}
\numberwithin{theorem}{section}

\numberwithin{equation}{section}

\newcommand{\La}{\mathcal{L}_a}

\newcommand{\C}{{\mathbb{C}}}
\newcommand{\R}{{\mathbb{R}}}

\let\Re=\undefined\DeclareMathOperator*{\Re}{Re}
\let\Im=\undefined\DeclareMathOperator*{\Im}{Im}

\renewcommand{\L}{\mathcal{L}_a}

\def\leq{\leqslant}
\def\geq{\geqslant}
\def\ge{\geqslant}
\def\le{\leqslant}
\begin{document}

\title[Focusing NLH with inverse square potential]{Focusing  nonlinear Hartree equation with inverse-square potential}
\author[Y. Chen]{Yu Chen}
\address{Graduate School of China Academy Of Engineering Physics,  \ Beijing, \ China, \ 100089, }
\email{chenyu17@gscaep.ac.cn}

\author[J. Lu]{Jing Lu}
\address{College of Science, China Agricultural University, \ Beijing, \ China, \ 100193, }
\email{lujing326@126.com}

\author[F. Meng]{Fanfei Meng}
\address{Graduate School of China Academy of Engineering Physics,  \ Beijing, \ China, \ 100089, }
\email{mengfanfei17@gscaep.ac.cn}

\subjclass[2000]{Primary 35Q55; Secondary 47J35}

\date{\today}

\keywords{Hartree equation, inverse-square potential, scatter,  Morawetz estimate. }
\maketitle

\begin{abstract}
In this paper, we consider the scattering theory of the radial solution to  focusing energy-subcritical Hartree equation with inverse-square potential in the energy space $H^{1}(\mathbb{R}^d)$ using the method from \cite{Dodson2016}. The main difficulties are the equation  is \emph{not} space-translation invariant and the nonlinearity is non-local. Using the radial Sobolev embedding and a virial-Morawetz type estimate we can exclude the concentration of mass near the origin. Besides, we can overcome the weak dispersive estimate when $a<0$, using the dispersive estimate established by \cite{zheng}.
\end{abstract}

\maketitle
\section{Introduction}
We consider the energy-subscritical Hartree equation with inverse-square potential:
\begin{equation}\tag{$\text{NLH}_a$}
\left\{
\begin{aligned}
&iu_{t}=\mathcal{L}_{a}u-(|\cdot|^{-\gamma}\ast|u|^{2})u, ~~~~ (t,x) \in \mathbb{R} \times \mathbb{R}^{d} \\
& u(0,x) = u_0(x), ~~~~ x \in \mathbb{R}^{d}.
\end{aligned}
\right.
\label{problem-eq: NLS-H}
\end{equation}
where $ u: \mathbb{R}\times \mathbb{R}^{d}\rightarrow \mathbb{C} $, $\mathcal{L}_{a}=-\Delta+\frac{a}{|x|^{2}}, a>-(\tfrac{d-2}2)^2$, $ 2 < \gamma < \min\{4,  d\}$, and $ * $ denotes the convolution in $ \mathbb{R}^{d} $.

Solutions to \eqref{problem-eq: NLS-H} conserve the \emph{mass} and \emph{energy}, defined respectively by
\begin{eqnarray}
 \nonumber 
  M(u) & = & \frac{1}{2}\int_{\mathbb R^{d}} |u|^{2}\equiv M(u_{0}),   \\
  E(u) & = & H(u) - P(u)\equiv E(u_{0}),
\end{eqnarray}
where
\[
\begin{aligned}
  H(u) & = \frac{1}{2}\int_{\mathbb R^{d}}\left(|\nabla u(x)|^{2} + \frac{a}{|x|^{2}} |u(x)|^{2}\right) dx,   \\
  P(u) & = \frac{1}{4} \iint_{\mathbb R^{d} \times \mathbb R^{d}} \frac{|u(x)|^{2}|u(y)|^{2}}{|x-y|^{\gamma}}\; dx dy.
\end{aligned}
\]
Note that, if $ a=0 $, then \eqref{problem-eq: NLS-H} reduces to the standard nonlinear Hartree equation:
\begin{equation}\label{nls0}\tag{$\text{NLH}_0$}
(i\partial_t+\Delta) u = \mu(|\cdot|^{-\gamma}\ast|u|^{2})u, ~~ \mu=\pm1.
\end{equation}
where $\mu=-1$ is called focusing case, and $\mu=+1$ is called defocusing.

Similar as \eqref{nls0},  the equation \eqref{problem-eq: NLS-H} enjoys the scaling symmetry
\[
  u(t,  x)\to u^{\lambda}(t,  x):=\lambda^{\frac{d+2-\gamma}{2}}u(\lambda^{2}t,  \lambda x).
\]
This symmetry identifies $ \dot{H}^{s_{c}}_{x}(\mathbb{R}^{d}) $ as the scaling-critical space of initial data,  where $s_{c}=\frac{\gamma}{2}-1$. The case we consider is the energy-subcritical problem,  which corresponds to $ 0<s_{c}<1 $.


As we know, a large amount of work has been devoted to the study of the scattering theory about the dispersive equations.
The scattering theory for the equation \eqref{nls0} with $ f(u) = \mu (|\cdot|^{-\gamma} \ast|u|^{2})u, 0 < \gamma \leq \min\{4,d\} $ has been studied in \cite{MXZ1,GV,8,Na1}. Concerning the energy-subcritical case $ 2 < \gamma < \min\{4,d\} $, using the method of Morawetz and Strauss \cite{MS}, J. Ginibre and G. Velo \cite{GV} developed the scattering theory in the energy space. Nakanishi \cite{Na1} improved the results by a new Morawetz estimate which doesn't depend on nonlinearity. In \cite{MXZ1}, C. Miao, G. Xu and L. Zhao obtained the small data scattering result for the energy-critical case in the energy space.  For the defocusing case $ \mu = 1 $: C. Miao, G. Xu and L. Zhao \cite{8} took advantage of a new kind of the localized Morawetz estimate, which is also independent of nonlinearity, to rule out the possibility of the energy concentration at origin and established the scattering results in the energy space for the radial data in dimension $ d \geq 5 $. For the focusing case $ \mu = -1 $, the dynamics of the solution to \eqref{nls0} will be more complex: In mass-critical case, C. Miao, G. Xu and L. Zhao in \cite{6} have proved the blow-up solution in finite time whose mass equals the mass of ground state must be a pseudo-conformal transformation of the ground state. In energy-critical case, D. Li, C. Miao and X. Zhang in \cite{7} established the scattering theory of the maximal lifespan interval solution whose energy is less than the energy of the ground state.

There are also a lot of results about scattering and blow-up theory for the nonlinear Schr\"{o}dinger equation with inverse-square potential (NLS$_{a}$). For the energy-critical case, R. Killip, C. Miao, M. Visan, J. Zhang and J. Zheng in  \cite{1} obtained the  scattering theory of the solution to (NLS$_{a}$) with the defocusing case in $d=3$  and focusing case in $ d \geq 3 $, using the concentration-compactness argument which is firstly introduced by R. Killip and M. Visan \cite{KV}. For the 3D cubic focusing case, R. Killip, J. Murphy, M. Visan and J. Zheng in \cite{2} showed the scattering and blow-up theory of the solution to (NLS$_{a}$). Soon, J. Lu, C. Miao and J. Murphy in \cite{3} established the scattering theory of the solution in any dimension, with the help of concentration-compactness argument.

The above progress in studying the scattering result of (NLS$_{a}$) or \eqref{nls0} is due to the approach based on the  concentration compactness argument to provide a linear profile decomposition. Recently, B. Dodson and J. Murphy in \cite{Dodson2016} used a new idea to simplify the proof of scattering below the ground state for the 3D radial focusing cubic NLS. The idea is that the virial-Morawetz type estimate established by Ogawa-Tsutsumi \cite{Ogawa1991} is directly applicable to arbitrary global solutions below the ground state energy, thanks to its variational characterization.
Later, C. Sun, H. Wang, X. Yao, J. Zheng \cite{SWYZ} adapted the strategy in \cite{Dodson2016} to prove the scattering of radial solutions below sharp threshold for certain focusing fractional NLS with cubic nonlinearity. F. Meng \cite{meng} give a simple proof about the result in \cite{Gao2010}, using the strategy in \cite{Dodson2016} that avoids the use of concentration compactness. Recently, J. Zheng in \cite{zheng} utilized the method in \cite{Dodson2016} to establish the radial scattering result for the focusing (NLS$_{a}$) with radial initial data. The new ingredient in \cite{zheng} is to establish the dispersive estimate for radial functions which will be used  to overcome the weak dispersive estimate when $a<0$ in our paper.

Motivated by the aforementioned papers, our aim is to adapt the method in \cite{Dodson2016} to prove the scattering result of energy-subcritical Hartree equation with inverse-square potential in the energy space $H^1(\mathbb{R}^d)$.

Our  main result in this paper  is follows:
\begin{theorem}[Scattering]\label{T:main2} Assume $ a $, $ d $ and $ \gamma $ satisfy
\begin{align}\label{important}
\begin{cases}
a>-\tfrac{(d-2)^{2}}{4},
 \begin{cases}
 ~2<\gamma\leq 3, ~ d=3;\\
 ~2<\gamma\leq \min\{\frac{d}2,3\}, ~ d>3.
 \end{cases} \\
a>\big(\tfrac{\gamma-3}{2}\big)^2-\tfrac{(d-2)^{2}}{4},
 \begin{cases}
 ~3<\gamma <4, ~ d=3,4;\\
 ~3<\gamma< \min\{\frac{d}2,4\}, ~ d>4.
 \end{cases}
\end{cases}
\end{align}
Given $u_0\in H_x^1(\R^d)$ be radial and satisfy $ M( u_0 )^{\frac{4-\gamma}{2}} E( u_0 )^{\frac{\gamma-2}{2}}  < M( Q )^{\frac{4-\gamma}{2}}  E( Q )^{\frac{\gamma-2}{2}} $ and $ \Vert u_0 \Vert^{\frac{4-\gamma}{2}} _{L_{x}^2(\mathbb{R}^d)} \Vert u_0 \Vert^{\frac{\gamma-2}{2}}_{\dot{H}_{x}^1(\mathbb{R}^d)} \leqslant \Vert Q \Vert^{\frac{4-\gamma}{2}} _{L_{x}^2(\mathbb{R}^d)} \Vert Q \Vert^{\frac{\gamma-2}{2}}_{\dot{H}_{x}^1(\mathbb{R}^d)} $, where $ Q $ is the solution to
$
- \mathcal{L}_{a} Q + Q = -(\vert \cdot \vert^{-\gamma} * \vert Q \vert^{2}) Q.
$
\label{main}

Then the solution of \eqref{problem-eq: NLS-H} is global and scatters in $ H_{x}^1 $, i.e. there exists $u_{\pm}\in H_x^1(\mathbb{R}^d)$ such that
\begin{align*}
\lim_{t\rightarrow\pm\infty}\|u(t)-e^{-it\mathcal{L}_{a}}u_{\pm}\|_{H_x^1(\mathbb{R}^d)}=0.
\end{align*}
\end{theorem}

\begin{remark}
The restriction \eqref{important}  comes from the local well-posedness theory Theorem \ref{T:LWP}, where we should verify the equivalence of Sobolev spaces.
\end{remark}

Compared with NLS with inverse square potential, our difficulty lies in the fact that Hartree equation  has a nonlocal term. The proof of Theorem \ref{main} consists of two steps: Firstly, we prove a certain decay estimate which can be deduced from an improved priori Morawetz estimate.

\begin{proposition}[Improved Morawetz estimate]\label{decay}
Let $ \phi, Q $ be as in Theorem \ref{main}. For any $ \epsilon > 0 $, there exist $ T \in (0, +\infty), R \in (0, +\infty) $ such that if $ u: \mathbb{R}_{t} \times \mathbb{R}_{x}^d \to \mathbb{C} $ is a radial solution to \eqref{problem-eq: NLS-H} satisfying
\begin{equation}
\Vert u \Vert_{L_{t}^{\infty}H_{x}^{1}(\mathbb{R}\times\mathbb{R}^d)} \leqslant E,
\label{ass}
\end{equation}
then
\begin{equation}
\tfrac{1}{T}\int_0^T P(\chi_{R} u) \mathrm{d}t \lesssim_{\delta, u} \tfrac{1}{T} + \tfrac{1}{R}.
\label{Me}
\end{equation}
\end{proposition}

Secondly, we establish a scattering criterion using the method from \cite{Dodson2016}.

\begin{proposition}[Scattering criterion]\label{scat}
 Let $ u, E $ be as in \eqref{ass}, and suppose \eqref{ass} holds.
For $ \epsilon > 0 $ be as in Theorem \ref{decay}, there exist $ R = R (E) > 0 $ such that if
\begin{equation}
\liminf_{t \to + \infty} \int_{\vert x \vert \leqslant R} \vert u(t, x) \vert^2 \mathrm{d}x \leqslant \epsilon^2,
\label{L2dl}
\end{equation}
then $ u $ scatters forward in time.
\end{proposition}
Combing the two steps, we can easily obtain the desired scattering result.

\section{Preliminaries}
We mark $ A \lesssim B $ to mean there exists a constant $ C > 0 $ such that $ A \leqslant CB $. We indicate dependence on parameters via subscripts, e.g. $ A \lesssim_{u} B $ indicates $ A \leqslant CB $ for some $ C = C(u) > 0 $. We write $ L_{t}^{q}L_{x}^{r} $ to denote the Banach space with norm
\begin{equation*}
\Vert u \Vert_{L_{t}^{q}L_{x}^{r}(\mathbb{R} \times \mathbb{R}^d)} := \left( \int_{\mathbb{R}} \left( \int_{\mathbb{R}^d} \vert u(t, x) \vert^{r} \mathrm{d}x \right) ^{\tfrac{q}{r}} \mathrm{d}t \right)^{\tfrac{1}{q}},
\end{equation*}
with the usual modifications when $ q $ or $ r $ are equal to infinity, or when the domain $ \mathbb{R} \times \mathbb{R}^d $ is replaced by space-time slab such as $ I \times \mathbb{R}^d $. We use $ (q, r) \in \Lambda_{s} $ to denote $ q \geqslant 2 $ and the pair satisfying
\begin{equation*}
\tfrac{2}{q} = d (\tfrac{1}{2} - \tfrac{1}{r}) - s.
\end{equation*}
Define that
\[
  \begin{aligned}
    \Vert f \Vert_{\dot{H}_{a}^{s_{c},r}(\mathbb{R}^{d})} =& ~ \Vert (\sqrt{\mathcal{L}_{a}})^{s_{c}}f \Vert_{L^{r}(\mathbb{R}^{d})} \\
    \Vert f \Vert_{H_{a}^{s_{c},r}(\mathbb{R}^{d})} =& ~ \Vert (\sqrt{1+\mathcal{L}_{a}})^{s_{c}}f \Vert_{L^{r}(\mathbb{R}^{d})}.
  \end{aligned}
\]
\subsection{Some useful inequalities}\label{ineq}

In this subsection, we show some important inequalities which are will be used frequently in the following sections.
\begin{lemma}[Riesz Rearrangement Inequality, \cite{LL}]\label{riesz}
We denote that $\;f^{\ast}$ is the radial non-increase symmetrical rearrangement of the function $f$,  that is to say,   denote $f^{\ast}$ as the rearrangement of $f$.  Then we have
\begin{equation*}
      \left|\iint_{\mathbb R^{d}\times\mathbb R^{d}} f(x)g(y)h(x-y)dxdy\right|
  \le
      \large\left| \iint_{\mathbb R^{d}\times\mathbb R^{d}} f^{\ast}(x)g^{\ast}(y)h^{\ast}(x-y)dxdy \large\right|
\end{equation*}
\end{lemma}

\begin{lemma}[Radial Sobolev Embedding, \cite{Tao}]\label{rad Sobole}
Let $ d \geq 3 $. For radial function $ f \in H^1(\mathbb{R}^d) $, there holds
\begin{align*}
\left\Vert \vert x \vert^{s} f \right\Vert_{L_x^{\infty}(\mathbb{R}^d)} \lesssim \Vert f \Vert_{H^1(\mathbb{R}^d)},
\end{align*}
where $ \tfrac{d}{2} - 1 \leq s \leq \tfrac{d-1}{2} $.
\end{lemma}

\begin{lemma}[Hardy-Littlewood-Sobolev Inequality, \cite{LL}]
If $ 1 < p,q < \infty $,  $ 0 < \alpha < d $ and $ \frac{1}{p} + \frac{1}{q} + \frac{\alpha}{d} = 2 $, we have
\begin{equation}\label{Hardy-Littlewood-Sobolev }
    \left|
        \iint_{\mathbb R^{d}\times \mathbb R^{d}}\frac{f(x)g(y)}{|x-y|^{\alpha}} dxdy
    \right|
        \lesssim \|f\|_{L^{p}(\mathbb R^{d})}\|g\|_{L^{q}(\mathbb R^{d})}
\end{equation}
\end{lemma}
\subsection{Harmonic analysis adapted to $ \La $} In this subsection, we describe some harmonic analysis tools adapted to the operator $ \La $. The primary reference for this section is \cite{11}.

Recall that by the sharp Hardy inequality, one has
\begin{equation}\label{iso}
\|\sqrt{\La}\, f\|_{L_x^2}^2 \sim \|\nabla f\|_{L_x^2}^2 ~ \textrm{for} ~ a > - (\tfrac{d-2}{2})^{2}.
\end{equation}
Thus, the operator $ \La $ is positive for $ a > - (\tfrac{d-2}{2})^{2} $. To state the estimates below, it is useful to introduce the parameter
\begin{equation}\label{rho}
\rho:=\tfrac{d-2}2-\bigr[\bigl(\tfrac{d-2}2\bigr)^2+a\bigr]^{\frac12}.
\end{equation}

We give the following result concerning equivalence of Sobolev spaces was established in \cite{11}; it plays an important role throughout this paper.

\begin{theorem}[Equivalence of Sobolev spaces,\cite{11}]\label{equivalence}
Fix $ d \ge 3 $, $ a\ge-(\frac{d-2}{2})^{2} $, and $ 0<s<2 $. If $ 1<p<\infty $ satisfies $ \frac{s+\sigma}{d}<\frac{1}{p}<min\{1,\frac{d-\sigma}{d}\} $, then
\begin{equation}
||(-\Delta)^{\frac{s}{2}}f||_{L^{p}}\le C(d,p,s)||\mathcal{L}_{a}^{\frac{s}{2}}f||_{L^{p}}\ \ \ \text{for all}\ f\in C_{c}^{\infty}(\mathbb{R}^{d}\setminus\{0\}).
\end{equation}
If $max\{\frac{s}{d},\frac{\sigma}{d}\}<\frac{1}{p}<min\{1,\frac{d-\sigma}{d}\}$,which ensures already that $1<p<\infty$,then
\begin{equation}
||\mathcal{L}_{a}^{\frac{s}{2}}f||_{L^{p}}\le C(d,p,s)||(-\Delta)^{\frac{s}{2}}f||_{L^{p}}\ \ \ \text{for all}\ f\in C_{c}^{\infty}(\mathbb{R}^{d}\setminus\{0\}).
\end{equation}
\end{theorem}
Next, we recall some fractional calculus estimates for powers of $\La$ due to Christ and Weinstein \cite{12}.
\begin{lemma}[Fractional product rule,\cite{12}]\label{product}
Fix $a>-(\frac{d-2}{2})^{2}+b$.Then for all $f,g\in C^{\infty}_{c}(\mathbb{R}^{d}\setminus\{0\})$ ,there exist $p_{a},p_{b}$ depends only on the range of $a$,then we have
\[
  ||\sqrt{\mathcal{L}_{a}}(fg)||_{L^{p}(\mathbb{R}^{d})}\le C(||\sqrt{\mathcal{L}_{a}}f||_{L^{p_{1}}(\mathbb{R}^{d})}||g||_{L^{p_{2}}(\mathbb{R}^{d})}+||f||_{L^{q_{1}}(\mathbb{R}^{d})}||\sqrt{\mathcal{L}_{a}}g||_{L^{q_{2}}(\mathbb{R}^{5})})
\]
for any exponents satisfying $p_{a}<p,p_{1},p_{2},q_{1},q_{2}<p_{b}$ and $\frac{1}{p}=\frac{1}{p_{1}}+\frac{1}{p_{2}}=\frac{1}{q_{1}}+\frac{1}{q_{2}}$.
\end{lemma}
Strichartz estimates for the propagator $e^{-it\La}$ were proved in \cite{BPSTZ}.  Combining these with the Christ--Kiselev lemma \cite{CK}, we arrive at the following:

\begin{theorem}[Strichartz estimates, \cite{BPSTZ}]\label{strichartz}
Fix $ a>-(\frac{d-2}{2})^{2} $, The solution $ u $ to
\[
iu_{t}=\mathcal{L}_{a}u+F
\]
on an interval $ I \ni t_{0} $ obeys
\begin{equation}
||u||_{L^{q}_{t}L^{r}_{x}(I\times\mathbb{R}^{d})}\le C(||u(t_{0})||_{L^{2}(\mathbb{R}^{d})}+||F||_{L^{\tilde{q}'}_{t}L^{\tilde{r}'}_{x}(I\times\mathbb{R}^{d})})
\end{equation}
whenever $(q,r),(\tilde{q},\tilde{r})\in \Lambda_{0}$, $2\le q,\tilde{q}\le\infty$, and $q\neq \tilde{q}$.
\end{theorem}

Now we show the dispersive estimate which will play a key role  in the proof of the scattering criterion.

\begin{theorem}[Dispersive estimate, \cite{zheng}]\label{dis}
Let $ f $ be a radial function,
\begin{enumerate}
\item[(i)] If $a>0$, then we have
\begin{align}\label{dis1}
\|e^{it\mathcal{L}_a}f\|_{L^\infty(\mathbb{R}^d)}\leq C|t|^{-\frac{d}2}\|f\|_{L_x^1(\mathbb{R}^d)}\end{align}
\item[(ii)] If $-\frac{(d-2)^2}{4}<a<0$, then
\begin{align}\label{dis2}
\|(1+|x|^{-\sigma})^{-1}e^{it\mathcal{L}_a}f\|_{L^\infty(\mathbb{R}^d)}\leq C\frac{1+|t|^{\sigma}}{|t|^{\frac{d}2}}\|(1+|x|^{-\sigma})f\|_{L_x^1(\mathbb{R}^d)}\end{align}
\end{enumerate}
\end{theorem}
In particularly,  then we can use the Riesz interpolation inequality to obtain
\begin{align}\label{dis3}
\|(1+|x|^{-\sigma})^{-1+\frac2p}e^{it\mathcal{L}_a}f\|_{L^p(\mathbb{R}^d)}\leq C|t|^{(-\frac{d}2+\sigma)(1-\frac2p)}\|(1+|x|^{-\sigma})^{1-\frac2p}f\|_{L_x^{p'}(\mathbb{R}^d)},
\end{align}
if $-\frac{(d-2)^2}{4}<a<0$.
\section{Local wellposedness}
In this section, we  state the local well-posedness for \eqref{problem-eq: NLS-H}.

\begin{theorem}[Local well-posedness]\label{T:LWP}
Assume $d\geq 3$,  $u_0\in H_x^1(\R^d)$,   $a$, $d$ and $\alpha$ satisfy \eqref{important}.
\begin{enumerate}
\item Then we have there exist $T=T(\|u_0\|_{H_a^1})>0$ and a unique solution $u:(t_0-T, t_0+T)\times\R^d\to\C$ to \eqref{problem-eq: NLS-H} with $u(t_0)=u_0$.
\item Given $A\ge 0$,  then there exists $\eta=\eta(A)$ such that $u_{0}\in H^{s_{c}}_{a}(\mathbb{R}^{d})$ obeys
\begin{equation}\label{local condition}
    ||(\sqrt{\mathcal{L}_{a}})^{s_{c}}u_{0}||_{L^{2}(\mathbb{R}^{d})}\le A\ \text{and}\ \ ||e^{it\mathcal{L}_{a}}u_{0}||_{L^{\alpha}_{t, x}}(I\times\mathbb{R}^{d})\le \eta.
\end{equation}
for some time interval $ I \ni 0 $. Then there is a unique strong solution $u$ to \eqref{problem-eq: NLS-H} on the time interval $ I $ such that
\[
||u||_{L^{\alpha}_{t, x}}(I\times\mathbb{R}^{d}) \le C \eta\ \ \text{and}\ \
||(\sqrt{\mathcal{L}_{a}})^{s_{c}}u||_{C_{t}L^{2}_{x} \cap L^{\alpha}_{t}L^{\beta}_{x}(I\times\mathbb{R}^{d})} \le CA.
\]
where $s_{c}=\frac{\gamma}{2}-1$, $\frac{1}{\alpha}=\frac{d-2s_{c}}{2(d+2)}$ and $\frac{1}{\beta}=\frac{1}{2}-\frac{d-2s_{c}}{d(d+2)}$.
$ a > -(\frac{d-2}{2})^{2} + b $, here
\begin{align}\label{b}
b= \max\{(\frac{6s_{c}+1}{d+2})^{2}, (\frac{d s_{c}-2}{d+2})^{2}, (\frac{-2d+2+6s_{c}}{d+2})^{2}\}.
\end{align}
\end{enumerate}
\end{theorem}

\begin{proof}
By time-translation symmetry we may choose $ t_0 =0 $. The proofs follow along standard lines using the contraction mapping principle; For convenience, all space-time norms in the proof will be taken over $ (-T, T) \times \R^d $. Define the map $ \Phi $ as
\[
[\Phi u](t) = e^{-it\mathcal{L}}u_0 + i\int_0^t e^{-i(t-s)\L}\bigl((|\cdot|^{-\gamma}\ast|u|^{2}) u(s)\bigr)\, ds.
\]
\begin{enumerate}
  \item We fix $\beta$ to be determined shortly and define the parameters
\[
(\tilde q, \tilde r) =(\tfrac{4\beta}{\beta-2},  \tfrac{2d\beta}{\beta d-\beta+2}), \quad s = \tfrac{1}{\beta}+\tfrac{\gamma-2}{2}.
\]
Fix $T>0$ and set $A=\|u_0\|_{H_a^1}$.
we need to show $\Phi$ is a contraction on the space
\[
B_T = \{u\in C_t H_a^1 \cap L_t^{\tilde q} H_a^{1, \tilde r}: \|u\|_{L^\infty_t H_a^1} \leq CA, \quad \|u\|_{L_t^{\tilde q} H_a^{1, \tilde r}}\leq CA\},
\]
which is complete with respect to the metric
\[
d(u, v) = \| u - v \|_{L_t^{\tilde q} L_x^{\tilde r}}.
\]

Let $u\in B_T$.  By Sobolev embedding and equivalence of Sobolev spaces,
\begin{align}\label{equa}
\|u\|_{L_t^{\tilde q} L_x^{\frac{2d}{d+1-\gamma}}} \lesssim \| u\|_{L_t^{\tilde q} H_a^{s, \tilde r}}  \lesssim A,
\end{align}
where we have used  $s\in[0, 1]$ and $H_a^{s, \tilde r}\hookrightarrow H_x^{s, \tilde r}$.

Thus,  by Strichartz,  Sobolev embedding,  and equivalence of Sobolev spaces,
\begin{align*}
\|\Phi u\|_{L^\infty_t H_a^1}\vee \|\Phi u\|_{L_t^{\tilde q} H_a^{1, \tilde r}} & \lesssim \|u_0\|_{H_a^1} + \|(|\cdot|^{-\gamma}\ast|u|^{2}) u\|_{L_t^2 H_x^{1, \frac{2d}{d+2}}} \\
& \lesssim A + |T|^{\tfrac{1}{\beta}} \| u\|_{L_t^{\tilde q} L_x^{\frac{2d}{d+1-\gamma}}}^{2} \|u\|_{L^\infty_t H_a^1}
\end{align*}
where we need $\beta>0$, and $(\tilde q, \tilde r)$ is an admissible pair.
Thus $\Phi:B_T\to B_T$,  provided $C$ is chosen sufficiently large and $T=T(\|u_0\|_{H_a^1})$ sufficiently small.  Similarly,  for $u, v\in B_T$,
\begin{align*}
&\|\Phi u-\Phi v\|_{L_t^{\tilde q} L_x^{\tilde r}}  \lesssim \| |u|^{\alpha}u - |v|^{\alpha} v\|_{L_{t}^{\frac{4\beta}{3\beta-2}}L_{x}^{\frac{2d\beta}{\beta+\beta d+2}}}\\
 & \lesssim |T|^{\frac{1}{\beta}}\|u-v\|_{L_t^{\tilde q} L_x^{\tilde r}} \bigl( \| u\|_{L_t^{\tilde q} L_x^{\frac{2d}{d+1-\gamma}}}^{2} + \|v\|_{L_t^{\tilde q} L_x^{\frac{2d}{d+1-\gamma}}}^{2} \bigr) \\
& \lesssim |T|^{\frac{1}{\beta}}A^2\|u-v\|_{L_t^{\tilde q} L_x^{\tilde r}},
\end{align*}
so that $\Phi$ is a contraction on $B_T$,  provided $T=T(\|u_0\|_{H_a^1})$ is sufficiently small.   This completes the proof.
\item It suffices to prove that $\Phi$  is a contraction on the (complete) space
\begin{align*}
    B:=\{u\in C_{t}H^{s_{c}}_{a}\cap L^{\alpha}_{t}H^{s_{c}, \beta}_{a}(I\times \mathbb{R}^{d}):||(\sqrt{\mathcal{L}_{a}})^{s_{c}}u||_{L^{\alpha}_{t}L^{\beta}_{x}}\le CA\  \ ||u||_{L^{\alpha}_{t, x}}\le 2\eta\}
\end{align*}
endowed with the metric
\[
    d(u, v):=||u-v||_{L^{\alpha}_{t}L^{\beta}_{x}}.
\]
where $\frac{1}{\alpha}=\frac{d-2s_{c}}{2(d+2)}$ and $\frac{1}{\beta}=\frac{1}{2}-\frac{d-2s_{c}}{d(d+2)}$.  The constant $C$ depends only on the dimension $d$ and $a$,  which reflects various constants in the Strichartz and Sobolev embedding inequalities.

By the Strichartz inequality,  (\ref{local condition}) and Lemma \ref{product},  weak young inequality and H\"{o}lder inequality,  for $u\in B$ we have
\[
\begin{aligned}
      &||(\sqrt{\mathcal{L}_{a}})^{s_{c}}\varphi(u)||_{C_{t}L^{2}_{x}\cap L^{\alpha}_{t}L^{\beta}_{x}}\\
      \le&CA+C||(\sqrt{\mathcal{L}_{a}})^{s_{c}}[(|\cdot|^{-\gamma}\ast|u|^{2})u]||_{L^{q'}_{t}L^{r'}_{x}}\\
      \le&CA+C||(\nabla)^{s_{c}}[(|\cdot|^{-\gamma}\ast|u|^{2})u]||_{L^{q'}_{t}L^{r'}_{x}}\\
      \le&CA+C|||\cdot|^{-\gamma}||_{L^{\frac{d}{\gamma}}_{x}(\mathbb{R}^{d})}||(\nabla)^{s_{c}}u||_{L^{\alpha}_{t}L^{\beta}_{x}}||u||^{2}_{L^{\alpha}_{t, x}}\\
      \le&CA+C||(\sqrt{\mathcal{L}_{a}})^{s_{c}}u||_{L^{\alpha}_{t}L^{\beta}_{x}}||u||^{2}_{L^{\alpha}_{t, x}}\\
      \le&CA+C\eta^{2}A\\
      \le&CA(\text{let}\ \eta\ \text{small enough})
\end{aligned}
\]
where $\frac{1}{q'}=\frac{3}{\alpha}=\frac{3(d-2s_{c})}{2(d+2)}$ and $\frac{1}{r'}=\frac{1}{2}+\frac{4-d+6s_{c}}{d(d+2)}$.

Similarly,  for $u\in B$,
\[
\begin{aligned}
    &||\Phi(u)||_{C_{t}L^{2}_{x}\cap L^{\alpha}_{t}L^{\beta}_{x}}\\
    \le&C||u_{0}||_{L^{2}_{x}(\mathbb{R}^{d})}+C||(|\cdot|^{-\gamma}\ast|u|^{2})u||_{L^{q'}_{t}L^{r'}_{x}}\\
    \le&C||u_{0}||_{L^{2}_{x}(\mathbb{R}^{d})}+C\eta^{2}||u||_{L^{\alpha}_{t}L^{\beta}_{x}}\\
    <&\infty.
\end{aligned}
\]
By $H^{s_{c}, \beta}\hookrightarrow L^{\alpha}$ and proceeding once more in a parallel manner shows
\[
\begin{aligned}
    &||\varphi(u)||_{L^{\alpha}_{t, x}}\\
    \le&C\eta+C||(\sqrt{\mathcal{L}_{a}})^{s_{c}}u||_{L^{\alpha}_{t}L^{\beta}_{x}}||u||^{2}_{L^{\alpha}_{t, x}}\\
    \le&C\eta+C(2\eta)^{2}2A\\
    \le&2\eta(\text{provided}\ \eta\ \text{is chosen sufficiently small}).
\end{aligned}
\]
To see that $\Phi$ is a contraction,  we argue analogously:
\[
\begin{aligned}
  ||\Phi(u)-\Phi(v)||_{L^{\alpha}_{t}L^{\beta}_{x}}\le& C||(|\cdot|^{-\gamma}\ast|u|^{2})u-(|\cdot|^{-\gamma}\ast|v|^{2})v||_{L^{q'}_{t}L^{r'}_{x}}\\
  \le&C||(|\cdot|^{-\gamma}\ast(\bar{u}-\bar{v})u)u||_{L^{q'}_{t}L^{r'}_{x}}\\
  +&||(|\cdot|^{-\gamma}\ast(u-v)\bar{v})u||_{L^{q'}_{t}L^{r'}_{x}}+||(|\cdot|^{-\gamma}\ast|u|^{2})(u-v)||_{L^{q'}_{t}L^{r'}_{x}}\\
  \le&C||u-v||_{L^{\alpha}_{t}L^{\beta}_{x}}(||u||^{2}_{L^{\alpha}_{t, x}}+||v||_{L^{\alpha}_{t, x}}||u||_{L^{\alpha}_{t, x}})\\
  \le&C\eta^{2}||u-v||_{L^{\alpha}_{t}L^{\beta}_{x}},
\end{aligned}
\]
If $\eta$ is small enough, we have $d(\Phi(u), \Phi(v))\le\frac{1}{2}d(u, v)$.
\end{enumerate}
\end{proof}

\begin{remark}
In Theorem \ref{T:LWP} (1), we need
\begin{enumerate}
                \item $\tfrac{1}{\beta}>0$,  this means $0<\beta<\infty$
                \item $s=\tfrac{1}{\beta}+\tfrac{\gamma-2}{2}\in[0, 1]$,  thus if $\gamma<3$,  then $\beta>2$; if $\gamma\geq 3$,  then $\beta>\frac{2}{4-\gamma}$
                \item $H_a^{s, \tilde r}\hookrightarrow H_x^{s, \tilde r}$,  thus
                \begin{align}
                \frac{\tfrac{1}{\beta}+\tfrac{\gamma-2}{2}+\sigma}{d}<\frac{\beta d-\beta+2}{2\beta d}<\min\{1, \tfrac{d-\sigma}{d}\}\end{align}
                this means $\sigma<\min\{\frac{d+1-\gamma}{2}, \frac{d+1}{2}-\frac{1}{\beta}\}=\frac{d+1-\gamma}{2}$
                \item $(\tilde q, \tilde r)$ is an admissible pair,  i. e.  $\tfrac{4\beta}{\beta-2}>2$.  we need $\beta>2$
\end{enumerate}
We now choose
\[
2 \vee \frac{2}{4-\gamma} < \beta <\infty.
\]
The upper bound and the lower bound on $\beta$ guarantees that $(\tilde q, \tilde r)$ is an admissible pair and $s\leq 1$.  The conditions on $a$ in \eqref{important} guarantee that $\dot H_x^{s, \tilde r}$ is equivalent to $\dot H_a^{s, \tilde r}$.

In Theorem \ref{T:LWP} (2),
\[
||\langle\nabla\rangle^{s} f||_{L^{p}}\le C(d, p, s)||\mathcal{L}_{a}^{\frac{s}{2}}f||_{L^{p}}\ \ \ \text{for all}\ f\in C_{c}^{\infty}(\mathbb{R}^{d}\setminus\{0\}).
\]
and the inverse inequality provided $p=r'=\frac{1}{2}+\frac{4-d+6s_{c}}{d(d+2)}$.
It can be calculated as follow.
\[
  \left\{
    \begin{aligned}
    &\frac{s_{c}+\sigma}{d}<\frac{1}{2}-\frac{d-2s_{c}}{d(d+2)}<\min\{1, \frac{d-\sigma}{d}\}\\
    &\max\{\frac{s_{c}}{d}, \frac{\sigma}{d}\}<\frac{1}{2}+\frac{4-d+6s_{c}}{d(d+2)}<\min\{1, \frac{d-\sigma}{d}\}
    \end{aligned}
  \right.
\]
it means

$\sigma<\min\{\frac{d^{2}-2s_{c}d}{2(d+2)}, \frac{d^{2}+8+12s_{c}}{2(d+2)}, \frac{d^{2}+4d-8-12s_{c}}{2(d+2)}\}$.

By the definition of $\sigma$,  we concluded that $a>-(\frac{d-2}{2})^{2}+b, $ where
$b$ satisfies \eqref{b}.
\end{remark}

\section{Variational Characterization}
In this section, we are in the position to give the variational characterization for the sharp Gargliardo-Nirenberg inequality. Firstly, We  will show the existence of the Ground state. As a corollary, we obtain  the sharp Gargliardo-Nirenberg inequality. Then we use the properties of the ground state to establish the Coercivity condition which will be used in the proof of Morawetz Estimate \eqref{Me}.

\begin{proposition}[The Existence of Ground State]\label{ground state}
The minimal $J_{\min}$ of the nonnegative funtional
\[
    J(u) : = (M(u))^{\frac{4-\gamma}{2}}(H(u))^{\frac{\gamma}{2}}(P(u))^{-1},  \qquad  u\in H^{1}(\mathbb R^{d}\setminus \{0\})
\]
are attained at a point $ W $, whose expression has to be in the form of $ W(x) = e^{i\theta} m Q (nx) $, where $ m, n > 0 $, $ \theta \in \mathbb R $, and $ Q\neq 0 $ is the non-negative radial solution of the equation
\begin{equation}\label{eq:ground-state}
(-\Delta + a |x|^{-2})Q + Q = (|\cdot|^{-\gamma} \ast |Q|^{2}) Q,
\end{equation}
where $ -(\tfrac{d-2}2)^{2} < a < 0 $ and $ 2 < \gamma < \min\{4,d\} $.

If $ Q \ge 0 $ is a non-negative radial solution of the equation \eqref{eq:ground-state} such that $J(Q)=J_{\min}$, then $ Q $ is called a {\bf Ground state}. The sets of all ground states is denotes as $\mathcal G$. All ground states share the same mass, denoted as $M_{gs}$.
\end{proposition}

Before proving the proposition, we show a primary lemma.

\begin{lemma}\label{lem:Sobolev-Lv}
If
\[
    \lim_{n\to\infty} \| u_{n} - u\|_{L^{\frac{2d}{d-\frac{\gamma}{2}}}(\mathbb R^{d})}=0,
\]
we have
\[
    \lim_{n\to\infty}P(u_{n} - u)= 0
    \text{ and }
    \lim_{n\to\infty} P(u_{n}) = P(u).
\]
\end{lemma}

\begin{proof}
By Hardy-Littlewood-Sobolev inequality, we gain
\begin{equation}\label{formula:Hardy-L-S in fact}
\begin{aligned}
  \left| \iint_{\mathbb R^{d}\times R^{d}} \frac{f_{1}(x)f_{2}(x)f_{3}(y)f_{4}(y)}{|x-y|^{\gamma}}dxdy \right|
  & \lesssim
  \|f_{1}f_{2}\|_{L^{\frac{d}{d-\frac{\gamma}{2}}}(\mathbb R^{d})} \cdot
  \|f_{3}f_{4}\|_{L^{\frac{d}{d-\frac{\gamma}{2}}}(\mathbb R^{d})} \\
  & \le \prod_{k=1}^{4}\|f_{k}\|_{L^{\frac{2d}{d-\frac{\gamma}{2}}}(\mathbb R^{d})}
\end{aligned}
\end{equation}
Therefore,
\[
    P(u_{n}- u) \lesssim \|u_{n}- u \|_{L^{\frac{2d}{d-\frac{\gamma}{2}}}(\mathbb R^{d})}^{4} \to 0,
    \text{ when }
    n \to \infty.
\]

Note that
\[
    |u(x)|^{2} |u(y)|^{2} - |v(x)|^{2} |v(y)|^{2}
  =
     (|u(x)|^{2}-|v(x)|^{2}) |u(y)|^{2} +  |v(x)|^{2} ( |u(y)|^{2}-|v(y)|^{2} )
\]

By \eqref{formula:Hardy-L-S in fact}, we get
\[
\begin{aligned}
  \left|P(u)-P(v)\right|
  \lesssim
  &\||u|^{2}-|v|^{2}\|_{L^{\frac{d}{d-\frac{\gamma}{2}}}(\mathbb R^{d})}
   \left(\| u\|^{2}_{L^{\frac{2d}{d-\frac{\gamma}{2}}}(\mathbb R^{d})} +\|v\|^{2}_{L^{\frac{2d}{d-\frac{\gamma}{2}}}(\mathbb R^{d})}\right) \\
  \lesssim
    &
   \|u-v\|_{L^{\frac{2d}{d-\frac{\gamma}{2}}}(\mathbb R^{d})}
   \left(
    \|u\|^{3}_{L^{\frac{2d}{d-\frac{\gamma}{2}}}(\mathbb R^{d})}+ \|u-v\|^{3}_{L^{\frac{2d}{d-\frac{\gamma}{2}}}(\mathbb R^{d})}
   \right)
\end{aligned}
\]

It means that if
  \[
    \lim_{n\to\infty} \| u_{n} - u\|_{L^{\frac{2d}{d-\frac{\gamma}{2}}}(\mathbb R^{d})} = 0,
  \]
we have
  \[
    \lim_{n\to\infty} P(u_{n}) = P(u).
  \]
\end{proof}
Next we give the  Schwartz symmetrical rearrangement argument about the functional $ J $.

\begin{lemma}\label{lemma:rearrangement}
Assume $u\in H^{1}(\mathbb R^{d})$ is a non-radial function, denote $u^{\ast}$ as the Schwartz symmetrical rearrangement  of $u$, then  $u^{\ast}\ge0$,
\[
    J(u^{\ast}) < J(u).
\]
\end{lemma}

\begin{proof}
By the classical Schwartz symmetrical rearrangement argument,  we know that $u^{\ast}$ satisfies
\[
\gathered
     M(u)=M(u^{\ast}),   \\
     \|\nabla u^{*}\|_{L^{2}(\mathbb R^{d})} \le \|\nabla u\|_{L^{2}(\mathbb R^{d})}. \\
\endgathered
\]
By Lemma \ref{riesz}, we get
\[
    P(u)\le P(u^{\ast}).
\]
Since $u$ is nonradial, then we have $u\neq0$ and
\[
     \int_{\mathbb R^{d}}\frac{|u|^{2}}{|x|^{2}} < \int_{\mathbb R^{d}}\frac{|u^{*}|^{2}}{|x|^{2}}. \\
\]
Therefore,
\[
    J(u^{\ast})<J(u)
\]
holds.
\end{proof}

Now we will prove Proposition \ref{ground state}. By solving a minimization problem, the minimum  is attained at the ground state of the corresponding stationary equation.

\begin{proof}[The proof of Proposition $ \ref{ground state} $]
We need to show the minimum can be attained first.

Suppose that the non-zero function sequence $\{u_{n}\}$ is the minimal sequence of the functional $J$,  that is to say,
\[
  \lim_{n\to\infty} J(u_{n}) = \inf\left\{J(u) :   u\in H^{1}(\mathbb R^{d}\setminus \{0\})\right\},
\]
By Lemma \ref{lemma:rearrangement}, without loss of generality, we can assume $u_{n}$ is non-negative radial.

Note that for any $u\in H^{1}(\mathbb R^{d}\setminus \{0\})$,  $\mu , \nu >0$,  we have
\begin{equation}\label{eq:scaling4MHLJ}
\left\{
\begin{gathered}
  M(\mu u(\nu \cdot)) = \mu ^{2}\nu ^{-d}M(u),  \\
  H(\mu u(\nu \cdot)) = \mu ^{2}\nu ^{2-d}H(u),  \\
  P(\mu u(\nu \cdot)) = \mu ^{4}\nu ^{\gamma-2d}P(u),  \\
  J(\mu u(\nu \cdot)) = J(u).  \\
\end{gathered}
\right.
\end{equation}
Denote
\[
  v_{n}(x) = \frac{\left(M(u_{n})\right)^{\frac{d-2}{4}}}{\left(H(u_{n})\right)^{\frac{d}{4}}}  u_{n}\left(\left(\frac{M(u_{n})}{H(u_{n})}\right)^{\frac{1}{2}}x\right)
\]
Then,  $v_{n}$ is non-negative radial, and
\[
  \aligned
  M(v_{n})=H(v_{n})\equiv1,   \quad J(v_{n}) = J(u_{n}), \\
  \lim_{n\to\infty} J(v_{n})= \inf\left\{J(u) :  0\ne u\in H^{1}(\mathbb R^{d})\right\}.
  \endaligned
\]
Note that $v_{n}$ is bounded in $H^{1}_{\text{rad}}(\mathbb R^{d})$  and
\[
  H^{1}_{\text{rad}}(\mathbb R^{d})\hookrightarrow H^{s_{c}}_{\text{rad}}(\mathbb R^{d})\hookrightarrow\hookrightarrow L^{\frac{2d}{d-\frac{\gamma}{2}}}(\mathbb R^{d}),
\]
 then there exist a subsequence $v_{n_{k}}$ and $v^{\ast}\in H^{1}_{\text{rad}}(\mathbb R^{d})$, such that as $k\to\infty$,  we have $v_{n_{k}}\rightharpoonup v^{\ast}$ in $H^{1}(\mathbb R^{d})$ and $v_{n_{k}}\to v^{\ast}$ in $L^{\frac{2d}{d-\frac{\gamma}{2}}}(\mathbb R^{d})$.

By the weak low semi-continuity of the functional $M$ and $H$, we obtain
\[
  M(v^{\ast})\le 1,  H(v^{\ast})\le 1.
\]
Since $\|v_{n_{k}} - v^{\ast}\|_{L^{\frac{2d}{d-\frac{\gamma}{2}}}(\mathbb R^{d})}\to 0$, by Lemma \ref{lem:Sobolev-Lv}, we have
\[
   P(v^{\ast} ) =\lim_{k\to\infty} P(v_{n_{k}}).
\]
Therefore,
\[
\begin{aligned}
  J(v^{\ast})
  &= \frac{(M(v^{\ast}))^{\frac{4-\gamma}{2}} (H(v^{\ast}))^{\frac{\gamma}{2}}}{P(v^{\ast})} \le  \lim_{k\to\infty}\frac{1}{ P(v_{n_{k}})}
  \\
  &= \lim_{k\to\infty}J(v_{n_{k}})
   = \inf\left\{J(u) :  0\ne u\in H^{1}(\mathbb R^{d})\right\}.
\end{aligned}
\]
Thus we proved that the minimum can be attained.

\vskip1em

Next, consider the variational derivatives of $ M $, $ H $, $ P $: fix $ u \neq 0 $, for any $ \varphi \in H^{1}(\mathbb R^{d}) $,
\begin{align}\label{formula: variational derivative}
   \frac{d}{d\epsilon}\Big|_{\epsilon=0}M(u+\epsilon\varphi) =& \Re \int_{\mathbb R^d} u\bar\varphi,   \\
  \frac{d}{d\epsilon}\Big|_{\epsilon=0}H(u+\epsilon\varphi) =& \Re \int_{\mathbb R^d} \left(-\Delta + \frac{ a}{|x|^{2}}\right)u\cdot \bar\varphi, \\
  \frac{d}{d\epsilon}\Big|_{\epsilon=0}P(u+\epsilon\varphi) =& \Re\iint\frac{|u(y)|^{2}u(x)\bar\varphi(x)}{|x-y|^{\gamma}}dydx,
\end{align}
If the functional $J$  attains the minimum at $W$, then we have for any $\varphi\in H^{1}(\mathbb R^{d})$,
  \[
    \aligned
      0
      = \frac{d}{d\epsilon}\Big|_{\epsilon=0}J(W+\epsilon\varphi).
    \endaligned
  \]
It means that
\begin{equation*}
\aligned
  (\frac{4-\gamma}{2})P(W)H(W)\cdot W+ (\frac{\gamma}{2})P(W) M(W) \cdot \left(-\Delta W+\frac{a}{|x|^{2}}W\right)
  \\-M(W)H(W)\cdot (|\cdot|^{-\gamma}\ast W^{2})W=0,
\endaligned
\end{equation*}
i.e.
\[
           (-\Delta +a|x|^{-2})W+\alpha W= \beta(|\cdot|^{-2}\ast |W|^{2})W,
\]
where
\[
    \alpha= \frac{(4-\gamma)H(W)}{\gamma M(W)}, \quad \beta= \frac{2H(W)}{\gamma P(W)}.
\]
By a direct calculation, we know
\begin{align*}
  (-\Delta +a|\cdot|^{-2})\bigl[\mu u(\nu\cdot )\bigr](x) =& \mu\nu^{2} (-\Delta +a|\cdot|^{-2})[u](\nu\cdot ), \\
  \bigl[(|\cdot|^{-2}\ast |\mu W(\nu \cdot)|^{2})\mu W(\nu \cdot)\bigr](x) =& \mu^{3}\nu^{2-d}[(|\cdot|^{-2}\ast |W|^{2})W](\nu\cdot )
\end{align*}
Therefore,  $ Q $  is the solution of \eqref{eq:ground-state} using the scaling $ W(x) = \alpha^{\frac{d}{4}}\beta^{-\frac{1}{2}}Q(\sqrt\alpha x) $.

\vskip1em
Next we prove that if $ W $ is the minimal element, then $ W $ is radial and there exists a  constant $ \theta \in \mathbb R $ such that $ W = e^{i\theta}|W| $.

If $ W $ is non-radial, then by Lemma \ref{lemma:rearrangement}, $ J(W^{\ast}) \linebreak[3] < J(W) $, which is contradict to the minimality of $ W $. So $ W $ is radial.

Since $ J(|W|) \le J(W) $,  $ |W| $ is also a minimal element. Suppose that $ W(x) = e^{i\theta(x)}|W|(x) $, where $ \theta(x) $ is a real-valued function, then
\begin{equation}\label{eq:iii}
\aligned
    |\nabla W(x)|^{2}
    & = \Bigl|e^{i\theta(x)}|W|(x)\cdot i\nabla\theta(x)+ e^{i\theta(x)}(\nabla|W|)(x)\Bigr|^{2} \\
    & = |W(x)|^{2}|\nabla\theta(x)|^{2} + \Bigl|\nabla|W|(x)\Bigr|^{2}
\endaligned
\end{equation}
By the minimality of $ J(W) $, $ J(W)=J(|W|) $. But by $ M(W) = M(|W|) $ and $ P(W) = P(|W|) $, we have $ H(W) = H(|W|) $.
So $ \nabla \theta(x) \equiv 0 $ in \eqref{eq:iii}, thus $ \theta(x)\equiv\text{constant} $.
Therefore, $ W(x) = e^{i\theta} m Q(nx) $, where $ m, n > 0 $, $ \theta \in \mathbb R $, and $ Q \neq 0 $ is the non-negative non-zero radial solution of \eqref{eq:ground-state}.

\vskip1em
Finally, we prove that all ground states have the same mass.

For $ \lambda \in (0, \infty) $,
\begin{align*}
  & M( \lambda^{\alpha} Q(\lambda^{\beta}\cdot) ) + E( \lambda^{\alpha} Q(\lambda^{\beta}\cdot) )
   \\&=
  \lambda^{2\alpha-\beta d}M(Q) + \lambda^{2\alpha + 2\beta - \beta d}H(Q) - \lambda^{4\alpha + 2\beta -2\beta d}P(Q)
\end{align*}
Using the chain rules and variational derivatives \eqref{formula: variational derivative}, then letting  $\lambda=1$ in the left side, we can obtain
\[
   \aligned
   &\Re\int_{\mathbb R^{d}} \left( -\Delta + \frac{ a}{|x|^{2}}+1 - |\cdot|^{-2}\ast|Q|^{2} \right) Q\cdot\
   \overline{\left( \frac{d}{d\lambda}\Big|_{\lambda=1} \lambda^{\alpha}Q(\lambda^{\beta}x)  \right) } dx   \\
    =&(2\alpha-\beta d)M(Q) + (2\alpha + 2\beta - \beta d)H(Q) - (4\alpha + 2\beta -2\beta d)P(Q),
   \endaligned
 \]
Since  $Q$ satisfies \eqref{eq:ground-state}, we have
\[
  (2\alpha-\beta d)M(Q) + (2\alpha + 2\beta - \beta d)H(Q) - (4\alpha + 2\beta -2\beta d)P(Q)\equiv0, \quad\forall \alpha, \beta.
\]

This yields
\begin{equation}\label{Q}
   M(Q) = H(Q) = P(Q) = J(Q)=\inf\{ J(u): u\in H^{1}(\mathbb R^{d}\backslash\{0\})\}=:M_{gs}.
\end{equation}
\end{proof}

Using the above proposition, we can directly obtain the following corollary.
\begin{corollary}[Gagliardo-Nirenberg inequality]
\begin{equation}\label{GN-inequality}
P(u) \leqslant C_{GN} \Vert u \Vert^{4-\gamma}_{L_{x}^2(\mathbb{R}^d)}
\Vert u \Vert_{\dot{H}_{a}^1(\mathbb{R}^d)}^{\gamma},
\end{equation}
where $ C_{GN} = \frac{1}{4} M_{gs}^{-1} $.
The equality holds if and only if $ u \in H^{1}(\mathbb R^{d}) $ is a minimal element of functional $ J(u) $, that is to say $ u \in \mathcal G $, or $ u = 0 $.
\end{corollary}
Indeed, By \eqref{eq:scaling4MHLJ} and $W(x) =e^{i \theta} m Q(nx)$, we have
\begin{align*}
J(W) = & ~ J_{min} = (M(W))^{\frac{4-\gamma}{2}} (H(W))^{\frac{\gamma}{2}} (P(W))^{-1} \\
= & ~ ( m^{2} n^{-d} M(Q) )^{\frac{4-\gamma}{2}} ( m^{2} n^{2-d} H(Q) )^{\frac{\gamma}{2}} ( m^{4} n^{\gamma-2d} P(Q) )^{-1}\\
=& ~J(Q)= M_{gs}.
\end{align*}
Thus,
\begin{align*}
M_{gs} \leqslant & ~ J(u) = (M(u))^{\frac{4-\gamma}{2}} (H(u))^{\frac{\gamma}{2}} (P(u))^{-1}, ~ \forall u \in H^{1}(\mathbb R^{d} \setminus \{0\})
\end{align*}
so,
\begin{equation*}
\begin{aligned}
P(u) \leqslant & ~ M_{gs}^{-1} (M(u))^{\frac{4-\gamma}{2}} (H(u))^{\frac{\gamma}{2}} \\
= & ~ M_{gs}^{-1} \left( \frac{1}{2} \Vert u \Vert_{2}^{2} \right)^{\frac{4-\gamma}{2}} \left( \frac{1}{2} \Vert u \Vert_{\dot{H}_{a}^1(\mathbb{R}^d)}^{2} \right)^{\frac{\gamma}{2}} \\
= & ~ C_{GN} \Vert u \Vert_{2}^{4-\gamma} \Vert \nabla u \Vert_{\dot{H}_{a}^1(\mathbb{R}^d)}^{\gamma} \\
\end{aligned}
\end{equation*}
where $ C_{GN} = \frac{1}{4} M_{gs}^{-1} $. \eqref{GN-inequality} holds.

%

To study the properties of ground state,  we start with elliptic equations \eqref{gse}.
\begin{equation}
- \Delta Q +\frac{a}{|x|^2}Q+ Q = (\vert \cdot \vert^{-\gamma} * \vert Q \vert^{2})Q.
\label{gse}
\end{equation}
%

It is easy to get some basic relations between $ M(Q), E(Q), P(Q) $ and the norms of $ Q $.
Multiplying \eqref{gse} by $ Q $ and $ x \cdot \nabla Q $ respectively,  integrating by parts leads to
\begin{align}
\Vert Q \Vert_{\dot{H}_{a}^1(\mathbb{R}^d)}^2 + \Vert Q \Vert_{L_{x}^2(\mathbb{R}^d)}^2 &= P(Q).
\label{gse1}\\
\tfrac{d-2}{2} \Vert Q \Vert_{\dot{H}_{a}^1(\mathbb{R}^d)}^2 + \tfrac{d}{2} \Vert Q \Vert_{L_{X}^2(\mathbb{R}^d)}^2& = \tfrac{2d-\gamma}{4} P(Q).
\label{gse2}
\end{align}

Combining \eqref{gse1} and \eqref{gse2},  we get
\begin{equation*}
\Vert Q \Vert_{\dot{H}_{a}^1(\mathbb{R}^d)}^2 = \tfrac{\gamma}{4} P(Q),  ~~
\Vert Q \Vert_{L_{x}^2(\mathbb{R}^d)}^2=\tfrac{4-\gamma}{4}P(Q).
\end{equation*}
So
\begin{equation}
C_{GN} = \tfrac{P(Q)}{\Vert Q \Vert_{L_{x}^2(\mathbb{R}^d)}^{4-\gamma} \Vert Q \Vert_{\dot{H}_{a}^1(\mathbb{R}^d)}^\gamma}
= \tfrac{16}{(4-\gamma)^{\frac{4-\gamma}{2}}\gamma^{\frac{\gamma}{2}}P(Q)},
\label{CGN}
\end{equation}
Meanwhile,  using energy and mass conservations,  we have
\begin{equation}
\begin{aligned}
M(Q)^{\tfrac{4-\gamma}{2}} E(Q)^{\tfrac{\gamma-2}{2}} = & \Vert Q \Vert_{L_{x}^2(\mathbb{R}^d)}^{4-\gamma}
\left(\tfrac{1}{2} \Vert Q \Vert_{\dot{H}_{a}^1(\mathbb{R}^d)}^2 - \tfrac{1}{4} P(Q)\right)^{\tfrac{\gamma-2}{2}} \\
= & \Vert Q \Vert_{L_{x}^2(\mathbb{R}^d)}^{4-\gamma}
\left(\tfrac{\gamma}{4-\gamma} \Vert Q \Vert_{L_{x}^2(\mathbb{R}^d)}^2
- \tfrac{1}{4-\gamma}\Vert Q \Vert_{L_{x}^2(\mathbb{R}^d)}^2\right)^{\tfrac{\gamma-2}{2}} \\
=& \Vert Q \Vert_{L_{x}^2(\mathbb{R}^d)}^{4-\gamma}
\left(\tfrac{\gamma-1}{4-\gamma} \Vert Q \Vert_{L_{x}^2(\mathbb{R}^d)}^2\right)^{\tfrac{\gamma-2}{2}}.
\end{aligned}
\label{gse3}
\end{equation}
\begin{lemma}[Coercivity $I$]\label{coer1}
 If $ M(\phi)^{\tfrac{4-\gamma}{2}} E(\phi)^{\tfrac{\gamma-2}{2}} < (1 - \delta) M(Q)^{\tfrac{4-\gamma}{2}} E(Q)^{\tfrac{\gamma-2}{2}} $
and $ \Vert \phi \Vert^{\tfrac{4-\gamma}{\gamma-2}} _{L_{x}^2(\mathbb{R}^d)} \Vert \phi \Vert_{\dot{H}_{a}^1(\mathbb{R}^d)} \leqslant \Vert Q \Vert^{\tfrac{4-\gamma}{\gamma-2}}_{L_{x}^2(\mathbb{R}^d)} \Vert Q \Vert_{\dot{H}_{a}^1(\mathbb{R}^d)} $,  then there exists $ \delta^\prime = \delta^\prime (\delta) > 0 $ so that
\begin{equation}
\Vert u(t) \Vert^{\tfrac{4-\gamma}{\gamma-2}} _{L_{x}^2(\mathbb{R}^d)} \Vert u(t) \Vert_{\dot{H}_{a}^1(\mathbb{R}^d)}
\leqslant (1 - \delta^\prime) \Vert Q \Vert^{\tfrac{4-\gamma}{\gamma-2}} _{L_{x}^2(\mathbb{R}^d)} \Vert Q \Vert_{\dot{H}_{a}^1(\mathbb{R}^d)}
\label{coe1}
\end{equation}
for all $ t \in I $,  where $ u : I \times \mathbb{R}^d \to \mathbb{C} $ is the maximal-lifespan solution to \eqref{problem-eq: NLS-H}.  In particular,  $ I = \mathbb{R} $ and $ u $ is uniformly bounded in $ H_{x}^1(\mathbb{R}^d) $.
\end{lemma}
\begin{proof}
Setting
\begin{equation*}
y(t) = \tfrac{\Vert u \Vert^{\tfrac{4-\gamma}{\gamma-2}} _{L_{x}^2(\mathbb{R}^d)}
\Vert u \Vert_{\dot{H}_{a}^1(\mathbb{R}^d)}}
{\Vert Q \Vert^{\tfrac{4-\gamma}{\gamma-2}} _{L_{x}^2(\mathbb{R}^d)} \Vert Q \Vert_{\dot{H}_{a}^1(\mathbb{R}^d)}}
\in C(I).
\end{equation*}
Considering Gagliardo-Nirenberg inequality \eqref{GN-inequality},  we have
\begin{equation*}
\begin{aligned}
M(u)^{\tfrac{4-\gamma}{2}} E(u)^{\tfrac{\gamma-2}{2}} = & \Vert u \Vert_{L_{x}^2(\mathbb{R}^d)}^{4-\gamma}
\left(\tfrac{1}{2} \Vert u \Vert_{\dot{H}_{a}^1(\mathbb{R}^d)}^2 - \tfrac{1}{4} P(u)\right)^{\tfrac{\gamma-2}{2}} \\
\geqslant & \Vert u \Vert_{L_{x}^2(\mathbb{R}^d)}^{4-\gamma}
\left(\tfrac{1}{2} \Vert u \Vert_{\dot{H}_{a}^1(\mathbb{R}^d)}^2 - \tfrac{1}{4}
C_{GN} \Vert u \Vert^{4-\gamma}_{L_{x}^2(\mathbb{R}^d)}
\Vert u \Vert^{\gamma}_{\dot{H}_{a}^1(\mathbb{R}^d)}\right)^{\tfrac{\gamma-2}{2}} \\
\end{aligned}
\end{equation*}
Using \eqref{CGN} and \eqref{gse3},  we get
\begin{equation*}
\begin{aligned}
\left(\tfrac{M(u)^{\tfrac{4-\gamma}{2}} E(u)^{\tfrac{\gamma-2}{2}}}{M(Q)^{\tfrac{4-\gamma}{2}} E(Q)^{\tfrac{\gamma-2}{2}}}\right)^{\tfrac{2}{\gamma-2}}
\geqslant & \tfrac{(4-\gamma)^2\Vert u \Vert_{L_{x}^2(\mathbb{R}^d)}^{\tfrac{2(4-\gamma)}{\gamma-2}}
 \Vert u \Vert_{\dot{H}_{a}^1(\mathbb{R}^d)}^2}
{ 2\gamma(\gamma-1)\Vert Q \Vert_{L_{x}^2(\mathbb{R}^d)}^{\tfrac{2(4-\gamma)}{\gamma-2}}
 \Vert Q \Vert_{\dot{H}_{a}^1(\mathbb{R}^d)}^2}
- \tfrac{\Vert u \Vert_{L_{x}^2(\mathbb{R}^d)}^{\tfrac{2(4-\gamma)}{\gamma-2}}
C_{GN} \Vert u \Vert^{4-\gamma}_{L_{x}^2(\mathbb{R}^d)}
\Vert u \Vert^{\gamma}_{\dot{H}_{a}^1(\mathbb{R}^d)}}
{4 \Vert Q \Vert_{L_{x}^2(\mathbb{R}^d)}^{\tfrac{2(4-\gamma)}{\gamma-2}}
 \Vert Q \Vert_{\dot{H}_{a}^1(\mathbb{R}^d)}^2} \\
= & \tfrac{\gamma}{\gamma-2}y^{2}-\tfrac{2}{\gamma-2}y^{\gamma} \in (0,  1-\delta).
\end{aligned}
\end{equation*}
So there exist $ 0 < \rho < 1 $ and $ \sigma > 1 $ such that
\begin{equation*}
\text{either} ~~ y(t) \in (0,   \rho),
\text{or} ~~ y(t) \in (\sigma,  \left(\tfrac{\gamma}{2}\right)^{\tfrac{1}{\gamma-2}}),  ~~~~ \forall t \in I.
\end{equation*}
Taking into account the case of $ t=0 $,  we can easily know $ y(t) \in (0,  \rho) $.  Let $ \delta^{\prime} = 1 - \rho $,  and then \eqref{coe1} holds.
\end{proof}

\begin{lemma}
[Coercivity $II$]\label{coer2}
Suppose $ \Vert u \Vert^{\tfrac{4-\gamma}{\gamma-2}} _{L_{x}^2(\mathbb{R}^d)} \Vert u \Vert_{\dot{H}_{a}^1(\mathbb{R}^d)} < ( 1 - \delta ) \Vert Q \Vert^{\tfrac{4-\gamma}{\gamma-2}} _{L_{x}^2(\mathbb{R}^d)} \Vert Q \Vert_{\dot{H}_{a}^1(\mathbb{R}^d)}$,  than there exists $ \delta^\prime = \delta^\prime (\delta) > 0 $ such that
\begin{equation}
\Vert u \Vert_{\dot{H}_{a}^1(\mathbb{R}^d)}^2 - \tfrac{\gamma}{4} P(u) \geqslant \delta^\prime P(u).
\label{coe2}
\end{equation}
\end{lemma}

\begin{proof}
Firstly,  by Gagliardo-Nirenberg inequality and \eqref{coe1}
\begin{equation*}
\begin{aligned}
E(u) = & \tfrac{1}{2} \Vert u \Vert_{\dot{H}_{a}^1(\mathbb{R}^d)}^2 - \tfrac{1}{4} P(u) \\
\geqslant & \tfrac{1}{2} \Vert u \Vert_{\dot{H}_{a}^1(\mathbb{R}^d)}^2 - \tfrac{1}{4}
C_{GN} \Vert u \Vert^{4-\gamma}_{L_{x}^2(\mathbb{R}^d)}
\Vert u \Vert_{\dot{H}_{a}^1(\mathbb{R})^5}^{\gamma} \\
= & \tfrac{1}{2} \Vert u \Vert_{\dot{H}_{a}^1(\mathbb{R}^d)}^2
\left( 1 - \tfrac{1}{2}
C_{GN} \Vert u \Vert^{4-\gamma}_{L_{x}^2(\mathbb{R}^d)}
\Vert u \Vert^{\gamma-2}_{\dot{H}_{a}^1(\mathbb{R}^d)} \right) \\
> & \tfrac{1}{2} \Vert u \Vert_{\dot{H}_{a}^1(\mathbb{R}^d)}^2
\left( 1 - \tfrac{1}{2}
C_{GN} ( 1 - \delta )^{\gamma-2} \Vert Q \Vert^{4-\gamma}_{L_{x}^2(\mathbb{R}^d)} \Vert Q \Vert^{\gamma-2}_{\dot{H}_{a}^1(\mathbb{R}^d)} \right) \\
= & \tfrac{1}{2} \Vert u \Vert_{\dot{H}_{a}^1(\mathbb{R}^d)}^2
\left( 1 - \tfrac{1}{2}( 1 - \delta )^{\gamma-2} \right)
\end{aligned}
\end{equation*}

Secondly,  by the definition of $ E(u) $,
\begin{equation*}
\begin{aligned}
\Vert u \Vert_{\dot{H}_{a}^1(\mathbb{R}^d)}^2 - \tfrac{\gamma}{4} P(u)
= & \gamma E(u) - \tfrac{\gamma-2}{2} \Vert u \Vert_{\dot{H}_{a}^1(\mathbb{R}^d)}^2 \\
\geqslant & \tfrac{\gamma}{2} \Vert u \Vert_{\dot{H}_{a}^1(\mathbb{R}^d)}^2
\left( 1 - \tfrac{1}{2} ( 1 - \delta )^{\gamma-2} \right)
- \tfrac{\gamma-2}{2} \Vert u \Vert_{\dot{H}_{a}^1(\mathbb{R}^d)}^2 \\
= & \left( 1 - \tfrac{\gamma}{4}(1 - \delta)^{\gamma-2} \right) \Vert u \Vert_{\dot{H}_{a}^1(\mathbb{R}^d)}^2,
\end{aligned}
\end{equation*}
which yields
\begin{equation*}
\begin{aligned}
\Vert u \Vert_{\dot{H}_{a}^1(\mathbb{R}^d)}^2
\geqslant \tfrac{1}{(1 - \delta)^{\gamma-2}} P(u).
\end{aligned}
\end{equation*}
Furthermore,
\begin{equation*}
\Vert u \Vert_{\dot{H}_{a}^1(\mathbb{R}^d)}^2 - \tfrac{\gamma}{4} P(u)
\geqslant \left( \tfrac{1}{(1 - \delta)^{\gamma-2}} - \tfrac{\gamma}{4} \right) P(u)
= \delta^{\prime} P(u).
\end{equation*}
\end{proof}

\begin{lemma}
[Coercivity on balls]\label{coerb}
There exists $ R = R(\delta,  M(u),  Q) > 0 $ sufficiently large such that
\begin{equation}
\sup_{t \in \mathbb{R}} \Vert \chi_{R} u(t) \Vert^{\tfrac{4-\gamma}{\gamma-2}}_{L_{x}^2(\mathbb{R}^d)} \Vert \chi_{R} u(t) \Vert_{\dot{H}_{a}^1(\mathbb{R}^d)} < (1 - \delta) \Vert Q \Vert^{\tfrac{4-\gamma}{\gamma-2}} _{L_{x}^2(\mathbb{R}^d)} \Vert Q \Vert_{\dot{H}_{a}^1(\mathbb{R}^d)}.
\label{coe}
\end{equation}
In particular, by Lemma \ref{coer2}, there exists $ \delta^\prime = \delta^\prime(\delta) > 0 $ so that
\begin{equation}
\Vert \chi_{R} u(t) \Vert_{\dot{H}_{a}^1(\mathbb{R}^d)}^2 - \tfrac{\gamma}{4} P(\chi_{R} u) \geqslant \delta^\prime P(\chi_{R} u)
\label{coe3}
\end{equation}
uniformly for $ t \in \mathbb{R} $.
\end{lemma}

\begin{proof}
First note that
\begin{equation*}
\Vert \chi_{R} u(t) \Vert_{L_{x}^2(\mathbb{R}^d)} \leqslant \Vert u(t) \Vert_{L_{x}^2(\mathbb{R}^d)}
\end{equation*}
uniformly for $ t \in \mathbb{R} $.  Thus,  it suffices to consider the
$ \dot{H}_{a}^1(\mathbb{R}^d) $ term.  For this,  we will make use of the following identity:
\begin{equation*}
\int_{\mathbb{R}^d} \chi_{R}^2 \vert \nabla u \vert^2 \mathrm{d}x = \int_{\mathbb{R}^d} \vert \nabla ( \chi_{R} u ) \vert^2 + \chi_{R} \Delta (\chi_{R}) \vert u \vert^2 \mathrm{d}x,
\end{equation*}
which can be obtained by a direct computation.  In particular,
\begin{equation*}
\Vert \chi_{R} u \Vert_{\dot{H}_{a}^1(\mathbb{R}^d)}^2 \leqslant \Vert u \Vert_{\dot{H}_{a}^1(\mathbb{R}^d)}^2 + \mathcal{O}\left(  \tfrac{1}{R^2} M(u) \right).
\end{equation*}
Choosing $ R $ sufficiently large depending on $ \delta,  M(u) $ and $ Q $,  the result follows.
\end{proof}
\section{The Proof of Scattering Criterion}

Now We will follow the strategy in \cite{Dodson2016} to prove the scattering criterion (Proposition \ref{scat}) in this section.
 \begin{proof}
 Our proof is divided into four steps.

{\bf Step one. } We claim that if
\begin{equation}
u(t,  x) \in L_{t}^{4}L_{x}^{\tfrac{2d}{d+1-\gamma}}(\mathbb{R} \times \mathbb{R}^d),
\label{str1}
\end{equation}
then the solution of \eqref{problem-eq: NLS-H} is global and scatters,  i. e.
$ \exists ~ u_{+}(x) \in H_{x}^{1} (\mathbb{R}^{5}) $,  s. t.
\begin{equation}
\lim_{t\to + \infty} \Vert u(t) - e^{it\mathcal{L}_{a}} u_{+} \Vert_{H_{x}^{1}(\mathbb{R}^{5})} = 0.
\label{sca}
\end{equation}

Firstly,  using Strichartz estimates Theorem \ref{strichartz},
\begin{equation*}
\begin{aligned}
\left\Vert \int_{t_1}^{t_2}e^{-is\mathcal{L}_{a}}f(u(s)) \mathrm{d}s \right\Vert_{H_{x}^{1}(\mathbb{R}^d)}
= & \left\Vert e^{-it_2\mathcal{L}_{a}} \int_{t_1}^{t_2}e^{i(t_2 - s)\mathcal{L}_{a}}f(u(s)) \mathrm{d}s \right\Vert_{H_{x}^{1}(\mathbb{R}^d)} \\
= & \left\Vert \int_{t_1}^{t_2}e^{i( t_2 - s)\mathcal{L}_{a}}f(u(s)) \mathrm{d}s \right\Vert_{H_{x}^{1}(\mathbb{R}^d)} \\
\lesssim & \left\Vert (\vert \cdot \vert^{-\gamma} * \vert u \vert^{2})u
\right\Vert_{L_{t}^{2}W_{x}^{1,  \tfrac{2d}{d+2}}([t_1,  t_2]\times \mathbb{R}^{d})} \\
\leqslant & \left\Vert (\vert \cdot \vert^{-\gamma} * \vert u \vert^{2})u
\right\Vert_{L_{t}^{2}W_{x}^{1,  \tfrac{2d}{d+2}}([t_1,  + \infty)\times \mathbb{R}^d)},
\end{aligned}
\end{equation*}
where we have used the fact that $ e^{it\Delta} $ is a unitary group for any time $ t $.

By H\"older inequality and Hardy-Littlewood-Sobolev inequality,  the inner integral is
\begin{equation}
\begin{aligned}
& \left\Vert (\vert \cdot \vert^{-\gamma} * \vert u \vert^{2})u
\right\Vert_{W_{x}^{1,  \tfrac{2d}{d+2}}(\mathbb{R}^d)} \\
\leqslant & \left\Vert \vert \cdot \vert^{-\gamma} * \vert u \vert^{2} \right\Vert_{L_{x}^{d}(\mathbb{R}^d)} \Vert u \Vert_{H_{x}^{1}(\mathbb{R}^d)}
+ \Vert 2\Re(\overline{u} \nabla u) \Vert_{L_{x}^{\tfrac{2d}{1+\gamma}}(\mathbb{R}^d)}\Vert u \Vert_{L_{x}^{\tfrac{2d}{d+1-\gamma}}(\mathbb{R}^d)} \\
\leqslant & \left\Vert u \right\Vert_{L_{x}^{\tfrac{2d}{d+1-\gamma}}(\mathbb{R}^d)}^{2}\Vert u \Vert_{H_{x}^{1}(\mathbb{R}^d)} + 2\Vert u \Vert_{L_{x}^{\tfrac{2d}{d+1-\gamma}}(\mathbb{R}^d)}^{2}\Vert u \Vert_{H_{x}^{1}(\mathbb{R}^d)}.
\end{aligned}
\label{inner}
\end{equation}

Integrating with respect to time variables and using H\"older inequality,  so
\begin{equation*}
\begin{aligned}
 \left\Vert \int_{t_1}^{t_2}e^{-is\mathcal{L}_{a}}f(u(s)) \mathrm{d}s \right\Vert_{H_{x}^{1}(\mathbb{R}^d)}
\leqslant & \left\Vert (\vert \cdot \vert^{-\gamma} * \vert u \vert^{2})u
\right\Vert_{L_{t}^{2}W_{x}^{1,  \tfrac{2d}{d+2}}([t_1,  + \infty)\times \mathbb{R}^d)} \\
\leqslant & ~ 3 \left\Vert \Vert u \Vert_{L_{x}^{\tfrac{2d}{d+1-\gamma}}(\mathbb{R}^d)}^{2} \right\Vert_{L_{t}^{2}([t_1,  +\infty))}
\left\Vert \Vert u \Vert_{H_{x}^{1}(\mathbb{R}^d)}
\right\Vert_{L_{t}^{\infty}([t_1,  +\infty))} \\
= & ~ 3 \Vert u \Vert_{L_{t}^{4}L_{x}^{\tfrac{2d}{d+1-\gamma}}([t_1,  + \infty) \times \mathbb{R}^d)}^{2}\Vert u \Vert_{L_{t}^{\infty} H_{x}^{1}([t_1,  +\infty) \times \mathbb{R}^d)},
\end{aligned}
\end{equation*}
which means
\begin{equation}
\lim_{t_{1}, t_{2}\to+\infty} \left\Vert \int_{t_1}^{t_2}e^{-is\mathcal{L}_{a}}f(u(s)) \mathrm{d}s \right\Vert_{H_{x}^{1}(\mathbb{R}^d)} = 0
\label{cau}
\end{equation}
holds since \eqref{str1} holds.

Secondly,  if we set
\begin{equation*}
u_+=e^{-iT\mathcal{L}_{a}}u(T)-i\int_{T}^{+\infty}e^{-is\mathcal{L}_{a}}f(u(s))\mathrm{d}s,
\end{equation*}
we can get $ u_+ \in H_{x}^{1}(\mathbb{R}^d)$,   i. e.
$ \Vert u_+ \Vert_{H_{x}^{1}(\mathbb{R}^d)} < +\infty $.

Actually,  we have
\begin{equation*}
\Vert e^{-iT\mathcal{L}_{a}} u(T) \Vert_{H_{x}^{1}(\mathbb{R}^d)} = \Vert u(T) \Vert_{H_{x}^{1}(\mathbb{R}^d)} \leqslant \Vert u \Vert_{L_{t}^{\infty}H_{x}^{1}(\mathbb{R}\times\mathbb{R}^d)} < +\infty.
\end{equation*}
By Cauchy convergence criterion and \eqref{cau},  we also have
\begin{equation*}
\left\Vert\int_{T}^{\infty}e^{i(t-s)\La}f(u(s)) \mathrm{d}s \right\Vert_{H_{x}^{1}(\mathbb{R}^d)}
=\left\Vert\int_{T}^{\infty}e^{-is\mathcal{L}_{a}}f(u(s)) \mathrm{d}s \right\Vert_{H_{x}^{1}(\mathbb{R}^d)}
< +\infty.
\end{equation*}
Thus,  $ u_+ \in H_{x}^{1}(\mathbb{R}^d) $ holds by the definition of $ u_+ $.

Lastly,  we prove that \eqref{sca} holds.
By the Duhamel formula
\begin{equation}
u(t) = e^{i(t-T)\La} u(T) - i \int_{T}^{t} e^{i(t-s)\La} f(u(s)) \mathrm{d}s,
\label{Duh}
\end{equation}
we have
\begin{equation*}
u(t) - e^{it\La} u_+ = i \int_{t}^{+ \infty} e^{i(t-s)\La} f(u(s)) \mathrm{d}s.
\end{equation*}

Using the similar argument and \eqref{cau},
\begin{equation*}
\begin{aligned}
\Vert u(t) - e^{it\La} u_+ \Vert_{H_{x}^{1}(\mathbb{R}^d)}
= & \left\Vert \int_{t}^{+ \infty} e^{i(t-s)\La} f(u(s)) \mathrm{d}s \right\Vert_{H_{x}^{1}(\mathbb{R}^d)} \\
= & \left\Vert \int_{t}^{+ \infty} e^{-is\La} f(u(s)) \mathrm{d}s \right\Vert_{H_{x}^{1}(\mathbb{R}^d)} \\
\rightarrow & ~ 0,  ~~~~ ( t \rightarrow + \infty ).
\end{aligned}
\end{equation*}
Thus the claim holds.

{\bf Step two. } We boil down the problem further and assert that if
\begin{equation}
\lim_{T \to +\infty} \int_{T-l}^{T} \Vert f(u(s)) \Vert_{H_{x}^{1}(\mathbb{R}^d)} \mathrm{d}s
= 0
\label{nlt}
\end{equation}
holds, then \eqref{str1} holds.

Firstly,  by Duhamel formula,  we have
\begin{equation}
\begin{aligned}
e^{i(t-T)\La}u(T)
= & ~ e^{it\La}u_0 - i\int_{0}^{T} e^{i(t-s)\La}f(u(s)) \mathrm{d}s \\
= &:  e^{it\La}u_0 - F_1 - F_2,
\end{aligned}
\label{dec}
\end{equation}
where
\begin{equation*}
F_1 = i\int_{0}^{T-l} e^{i(t-s)\La}f(u(s)) \mathrm{d}s,
\end{equation*}
and
\begin{equation*}
F_2 = i\int_{T-l}^{T} e^{i(t-s)\La}f(u(s)) \mathrm{d}s,
\end{equation*}
where $l$ will be decided later.

Next let us deal with the three terms in \eqref{dec} respectively. Now denote $p = \tfrac{2d}{d+1-\gamma}.$

Owing to \eqref{ass} and Strichartz estimates,  we know that
\begin{equation}
\lim_{T \to +\infty} \left\Vert e^{it\La}u_0 \right\Vert_{L_{t}^{4}L_{x}^{p}
([T,  + \infty)\times \mathbb{R}^d)} = 0.~
(4, p)\in \Lambda_{\frac{\gamma-2}{2}}
\label{F0}
\end{equation}

For $ F_1 $,  note that
\begin{align*}
\Vert F_1 \Vert_{L_{t}^{4}L_{x}^{p}([T,  + \infty)\times \mathbb{R}^d)} \leqslant \Vert F_1 \Vert_{L_{t}^{\infty}L_{x}^{p}([T,  + \infty)\times \mathbb{R}^d)}^{\tfrac{\gamma-2}{\gamma-1}}
\Vert F_1 \Vert_{L_{t}^{\tfrac{4}{\gamma-1}}L_{x}^{p}([T,  + \infty)\times \mathbb{R}^d)}^{\tfrac{1}{\gamma-1}}.~ (\tfrac{4}{\gamma-1}, p)\in \Lambda_{0}
\end{align*}
We will estimate the two factors respectively.

On one hand,
\begin{equation*}
\begin{aligned}
& \Vert F_1 \Vert_{L_{t}^{\tfrac{4}{\gamma-1}}L_{x}^{p}([T, +\infty)\times \mathbb{R}^d)} \\
= & \left\Vert e^{i(t-T+l)\La} \int_{0}^{T-l} e^{i(T-l-s)\La}f(u(s)) \mathrm{d}s \right\Vert_{L_{t}^{\tfrac{4}{\gamma-1}}L_{x}^{p}([T, +\infty)\times \mathbb{R}^d)} \\
\leqslant & \left\Vert e^{i(t-T+l)\La} u(T-l) \right\Vert_{L_{t}^{\tfrac{4}{\gamma-1}}L_{x}^{p} ([T, +\infty) \times \mathbb{R}^d)}
+ \left\Vert e^{it\La} u_0 \right\Vert_{L_{t}^{\tfrac{4}{\gamma-1}}L_{x}^{p} ([T, +\infty)\times \mathbb{R}^d)} \\
\lesssim & ~ \Vert u(T-l) \Vert_{L_{x}^{2}(\mathbb{R}^d)} + \Vert u_0 \Vert_{L_{x}^{2}(\mathbb{R}^d)} \\
\leqslant & ~ 2 \Vert u \Vert_{L_{t}^{\infty}H_{x}^{1}(\mathbb{R}\times \mathbb{R}^d)}.
\end{aligned}
\end{equation*}

On the other hand,  now our position is into  to estimate $ \Vert F_1 \Vert_{L_{t}^{\infty} L_{x}^{p}([T, +\infty) \times \mathbb{R}^d)}$.

If $d\leq 2\gamma$, then using dispersive estimate \eqref{dis3}, H\"older inequality and Hardy-Littlewood-Sobolev inequality,
\begin{equation*}
\begin{aligned}
\left\Vert F_1 \right\Vert_{L_{x}^{p}(|x|\leq R_1)}
\leqslant &  \int_{0}^{T-l}\left\Vert e^{i(t-s)\La}f(u(s)) \right\Vert_{L_{x}^{p}(|x|\leq R_1)}\mathrm{d}s \\
\leqslant &  \int_{0}^{T-l} \Vert(1+|x|^{-\alpha})^{-1}  e^{i(t-s)\La}[(|x|^{-\gamma} * \vert u \vert^{2})u] \Vert_{L_{x}^{\infty}(\mathbb{R}^d)} \mathrm{d}s\Vert(1+|x|^{-\alpha})\Vert_{L_{x}^{p}(|x|\leq R_1)} \\
\leqslant & R_1^{\tfrac{d+1-\gamma}{2}}\int_{0}^{T-l} \vert t-s \vert^{-\tfrac{d}{2}+\alpha} \Vert(1+|x|^{-\alpha})  [(|x|^{-\gamma} * \vert u \vert^{2})u] \Vert_{L_{x}^{1}(\mathbb{R}^d)} \mathrm{d}s \\
\leqslant & ~ R_1^{\tfrac{d+1-\gamma}{2}}\vert t-T+l \vert^{-\tfrac{d-2}{2}+\alpha},
\end{aligned}
\end{equation*}
where $\alpha=0, $ if $a\geq 0$; and $\alpha=\sigma, $ if $a< 0$.
and we have used
\begin{align*}
\Vert (|x|^{-\gamma} * \vert u \vert^{2})u \Vert_{L_{x}^{1}(\mathbb{R}^d)}\lesssim& \Vert (|x|^{-\gamma} * \vert u \vert^{2}) \Vert_{L_{x}^{2}(\mathbb{R}^d)}\|u\|_{L_x^2}\\
\lesssim& \|u^2\|_{L_x^{\tfrac{2d}{3d-2\gamma}}}\|u\|_{L_x^2}\\
=& \|u\|^2_{L_x^{\tfrac{4d}{3d-2\gamma}}}\|u\|_{L_x^2}\lesssim \Vert u \Vert^3_{L_{t}^{\infty}H_{x}^{1}},
\end{align*}
where we need $2\leq\tfrac{4d}{3d-2\gamma}\leq \tfrac{2d}{d-2}$ and $3d-2\gamma>0$,  which implies $d\leq 2\gamma$.
\begin{align*}
\Vert|x|^{-\sigma} [(|x|^{-\gamma} * \vert u \vert^{2})u] \Vert_{L_{x}^{1}(\mathbb{R}^d)}\lesssim \||x|^{-\sigma}u\|_{L_x^2}\||x|^{-\gamma} * \vert u \vert^{2}\|_{L_x^2}\\
\lesssim \||\nabla|^{\sigma}u\|_{L_x^2}\| u \|^{2}_{L_x^{\tfrac{4d}{3d-2\gamma}}}\lesssim \Vert u \Vert^3_{L_{t}^{\infty}H_{x}^{1}},
\end{align*}
where $ 0 < \sigma< 1 $ by the assumption $ \eqref{important} $.

For $ d > 2 \gamma $
\begin{equation*}
\begin{aligned}
 &\left\Vert F_1 \right\Vert_{L_{x}^{p}(|x|\leq R_1)}
\leqslant \int_{0}^{T-l}\left\Vert e^{i(t-s)\La}f(u(s)) \right\Vert_{L_{x}^{p}(|x|\leq R_1)}\mathrm{d}s \\
&\leqslant \int_{0}^{T-l} \Vert(1+|x|^{-\alpha})^{-\tfrac{2\gamma}{d}}  e^{i(t-s)\La}[(|x|^{-\gamma} * \vert u \vert^{2})u] \Vert_{L_{x}^{\tfrac{2d}{d-2\gamma}}(\mathbb{R}^d)} \mathrm{d}s\Vert(1+|x|^{-\alpha})^{\tfrac{2\gamma}{d}}\Vert_{L_{x}^{\tfrac{2d}{\gamma+1}}(|x|\leq R_1)} \\
&\leqslant R_1^{\tfrac{1+\gamma}{2}}\int_{0}^{T-l} \vert t-s \vert^{(-\tfrac{d}{2}+\alpha)\tfrac{2\gamma}{d}} \Vert(1+|x|^{-\alpha})^{\tfrac{2\gamma}{d}}  [(|x|^{-\gamma} * \vert u \vert^{2})u] \Vert_{L_{x}^{\tfrac{2d}{d+2\gamma}}(\mathbb{R}^d)} \mathrm{d}s \\
&\leqslant ~ R_1^{\tfrac{1+\gamma}{2}}\vert t-T+l \vert^{(-\tfrac{d}{2}+\alpha)\tfrac{2\gamma}{d}+1}.
\end{aligned}
\end{equation*}
where $ \alpha = 0 $, if $ a \geq 0 $; and $ \alpha = \sigma $, if $ a < 0 $.
and we have used
\begin{align*}
\Vert (|x|^{-\gamma} * \vert u \vert^{2})u \Vert_{L_{x}^{\tfrac{2d}{d+2\gamma}}(\mathbb{R}^d)}\lesssim& \Vert (|x|^{-\gamma} * \vert u \vert^{2}) \Vert_{L_{x}^{\tfrac{d}{\gamma}}(\mathbb{R}^d)}\|u\|_{L_x^{2}}\\
\lesssim& \|u^2\|_{L_x^{1}}\|u\|_{L_x^{2}}
=\|u\|^3_{L_x^{2}}\lesssim \Vert u \Vert^3_{L_{t}^{\infty}H_{x}^{1}}
\end{align*}
and
\begin{align*}\Vert |x|^{-\alpha\tfrac{2\gamma}{d}}  [(|x|^{-\gamma} * \vert u \vert^{2})u] \Vert_{L_{x}^{\tfrac{2d}{d+2\gamma}}(\mathbb{R}^d)}\lesssim&  \Vert (|x|^{-\gamma} * \vert u \vert^{2}) \Vert_{L_{x}^{\tfrac{d}{\gamma}}(\mathbb{R}^d)}\||x|^{-\alpha\tfrac{2\gamma}{d}}\cdot u\|_{L_x^{2}}\\
\lesssim& \Vert u \Vert^2_{L_{x}^{2}}\Vert u \Vert_{L_{t}^{\infty}H_{x}^{1}}\\
\lesssim& \Vert u \Vert^3_{L_{t}^{\infty}H_{x}^{1}}
\end{align*}

On the exterior of the ball, using the radial Sobolev embedding we have
\begin{align*}
\left\Vert F_1 \right\Vert_{L_{x}^{p}(|x|\geq R_1)}\lesssim \|F_1\|^{\tfrac{d+1-\gamma}{d}}_{L_x^2}\|F_1\|^{\tfrac{\gamma-1}{d}}_{L_x^{\infty}(|x|\geq R_1)}\lesssim R_1^{-\tfrac{d-1}{2}\tfrac{\gamma-1}{d}}.
\end{align*}
Thus we complete the estimate of $F_1$ after choosing $ R_1^{\tfrac{d+1-\gamma}{2}} \vert t-T+l \vert^{-\tfrac{d-2}{2}+\alpha} = R_1^{-\tfrac{d-1}{2} \tfrac{\gamma-1}{d}} $.

For $ F_2 $,  by Strichartz estimate,  interpolation inequality and inequality of arithmetic and geometric mean ( i. e.  $ \sqrt{ab} \leqslant \tfrac{a + b}{2},  ~~ \forall ~ a,  b \geqslant 0 $ ),  we know
\begin{equation*}
\begin{aligned}
 &\left\Vert \int_{T-l}^{T} e^{i(t-s)\La}f(u(s)) \mathrm{d}s \right\Vert_{L_{t}^{4}L_{x}^{\tfrac{2d}{d+1-\gamma}}([T,  + \infty)\times \mathbb{R}^d)}\\
&\leqslant \int_{T-l}^{T} \left\Vert e^{i(t-s)\La} f(u(s)) \right\Vert_{L_{t}^{4}\dot{W}_{x}^{\tfrac{\gamma-2}{2},  \tfrac{2d}{d-1}}([T,  + \infty)\times \mathbb{R}^d)} \mathrm{d}s \\
&\lesssim \int_{T-l}^{T} \left\Vert f(u(s)) \right\Vert_{\dot{W}_{x}^{\tfrac{\gamma-2}{2},  2}(\mathbb{R}^d)} \mathrm{d}s \\
&\leqslant \int_{T-l}^{T} \Vert f(u(s)) \Vert_{L_{x}^{2}(\mathbb{R}^d)}^{\theta} \Vert f(u(s)) \Vert_{\dot{H}_{x}^{1}(\mathbb{R}^d)}^{1-\theta} \mathrm{d}s \\
 &\leqslant ~ \tfrac{1}{2} \int_{T-l}^{T} \Vert f(u(s)) \Vert_{H_{x}^{1}(\mathbb{R}^d)} \mathrm{d}s.
\end{aligned}
\end{equation*}
By the assumption \eqref{nlt},  it is obvious to get
\begin{equation}
\lim_{T \to +\infty} \left\Vert \int_{T-l}^{T} e^{i(t-s)\La}f(u(s)) \mathrm{d}s \right\Vert_{L_{t}^{4}L_{x}^{\tfrac{2d}{d+1-\gamma}}([T,  + \infty)\times \mathbb{R}^d)} = 0.
\label{F2}
\end{equation}

Combining with \eqref{dec}-\eqref{F2},  we have
\begin{equation}
\lim_{T \to +\infty} \left\Vert e^{i(t-T)\La}u(T) \right\Vert_{L_{t}^{4}L_{x}^{\tfrac{2d}{d+1-\gamma}}
([T,  + \infty)\times \mathbb{R}^d)} = 0.
\label{str2}
\end{equation}

Secondly,  making use of Duhamel formula \eqref{Duh} again we can see clearly that
\begin{equation*}
\begin{aligned} \Vert u \Vert_{L_{t}^{4}L_{x}^{\tfrac{2d}{d+1-\gamma}}([T,  + \infty)\times \mathbb{R}^d)}
\leqslant & \left\Vert e^{i(t-T)\La}u(T) \right\Vert_{L_{t}^{4}L_{x}^{\tfrac{2d}{d+1-\gamma}}
([T,  + \infty)\times \mathbb{R}^d)} \\
& + \left\Vert \int_{T}^{t}e^{i(t-s)\La}(\vert \cdot \vert^{-\gamma} * \vert u \vert^{2})u(s)\mathrm{d}s \right\Vert_{L_{t}^{4}L_{x}^{\tfrac{2d}{d+1-\gamma}}([T,  + \infty)\times \mathbb{R}^d)}.
\end{aligned}
\end{equation*}

Until now,  we can deduce from \eqref{str2} that the first term is small as long as $ T $ is large enough,  so we only need to estimate the second term.  In fact,  making use of the similar  estimate to $ F_{2} $ term,  we have
\begin{equation*}
\begin{aligned}
& \left\Vert \int_{T}^{t}e^{i(t-s)\La}(\vert \cdot \vert^{-\gamma} * \vert u \vert^{2})u(s)\mathrm{d}s \right\Vert_{L_{t}^{4}L_{x}^{\tfrac{2d}{d+1-\gamma}}([T,  + \infty)\times \mathbb{R}^d)} \\
\leqslant & \left\Vert \int_{T}^{t}e^{i(t-s)\La}(\vert \cdot \vert^{-\gamma} * \vert u \vert^{2})u(s)\mathrm{d}s \right\Vert_{L_{t}^{4}\dot{W}_{x}^{\tfrac{\gamma-2}{2},  \tfrac{2d}{d-1}}([T,  + \infty)\times \mathbb{R}^d)} \\
\lesssim & \left\Vert (\vert \cdot \vert^{-\gamma} * \vert u \vert^{2}) u \right\Vert_{L_{t}^{2}\dot{W}_{x}^{\tfrac{\gamma-2}{2},  \tfrac{2d}{d+2}}([T,  + \infty)\times \mathbb{R}^d)} \\
\leqslant & \left\Vert ~ \tfrac{3}{2} \Vert u \Vert_{L_{x}^{\tfrac{2d}{d+1-\gamma}}(\mathbb{R}^d)}^{2} \Vert u \Vert_{H_{x}^{1}(\mathbb{R}^d)} \right\Vert_{L_{t}^{2}([T,  +\infty))} \\
\leqslant & ~ \tfrac{3}{2} \Vert u \Vert_{L_{t}^{4}L_{x}^{\tfrac{2d}{d+1-\gamma}}([T,  +\infty) \times \mathbb{R}^d)}^{2} \Vert u \Vert_{L_{t}^{\infty}H_{x}^{1}([T,  +\infty) \times \mathbb{R}^d)}.
\end{aligned}
\end{equation*}
where we have used Strichartz estimats,  interpolation inequality,  inequality of arithmetic and geometric mean and \eqref{inner}.

Thus by continuity method and \eqref{str2},  we can easily deduce that \eqref{str1} holds.

{\bf Step three. } We assert that \eqref{nlt} holds if
\begin{equation}
\Vert u(t,  x) \Vert_{L_{t}^{\infty}L_{x}^{\tfrac{2d}{d+2-\gamma}}([T-l,  T] \times \mathbb{R}^d)}
\lesssim \epsilon^{\tfrac{1}{2}}
\label{eps}
\end{equation}
holds.
More accurately,  the upper bound of $ \int_{T-l}^{T} \Vert f(u(s)) \Vert_{H_{x}^{1} (\mathbb{R}^d)} \mathrm{d}s  $ is $ \epsilon^{\tfrac{1}{4}} $ for $ T $ sufficiently large.


Firstly,   for the nonlinear term $ f(u) = - (\vert \cdot \vert^{-\gamma} * \vert u \vert^{2})u $, we have
\begin{equation*}
\begin{aligned}
\Vert f(u) \Vert_{H_{x}^{1}(\mathbb{R}^d)}
\leqslant & \Vert \vert \cdot \vert^{-\gamma} * \vert u \vert^2 \Vert_{L_{x}^d(\mathbb{R}^d)}
\Vert u \Vert_{W_{x}^{1,  \tfrac{2d}{d-2}}(\mathbb{R}^d)} \\
& + \Vert \vert \cdot \vert^{-\gamma} * ( 2 u \nabla u ) \Vert_{L_{x}^{\tfrac{2d}{\gamma-1}}(\mathbb{R}^{d})} \Vert u \Vert_{L_{x}^{\tfrac{2d}{d+1-\gamma}}(\mathbb{R}^{d})}\\
\leqslant & \Vert u^2 \Vert_{L_{x}^{\tfrac{d}{d+1-\gamma}}(\mathbb{R}^d)}
\Vert u \Vert_{W_{x}^{1,  \tfrac{2d}{d-2}}(\mathbb{R}^d)} + \Vert 2 u \nabla u \Vert_{L_{x}^{\tfrac{2d}{2d-1-\gamma}}(\mathbb{R}^d)}
\Vert u \Vert_{L_{x}^{\tfrac{2d}{d+1-\gamma}}(\mathbb{R}^d)} \\
\leqslant & ~ 3 \Vert u \Vert_{L_{x}^{\tfrac{2d}{d+1-\gamma}}(\mathbb{R}^d)}^2
\Vert u \Vert_{W_{x}^{1,  \tfrac{2d}{d-2}}(\mathbb{R}^d)}.
\end{aligned}
\end{equation*}
Substituting it into the integral to time and get
\begin{equation}
\begin{aligned}
& \int_{T-l}^{T} \Vert f(u(s)) \Vert_{H_{x}^{1} (\mathbb{R}^d)} \mathrm{d}s \\
\leqslant & ~ \Vert u \Vert_{L_{t}^{\infty}L_{x}^{\tfrac{2d}{d+2-\gamma}}([T-l,  T] \times \mathbb{R}^d)}
\Vert u \Vert_{L_{t}^2 L_{x}^{\tfrac{2d}{d-\gamma}}([T-l,  T] \times \mathbb{R}^d)}
\Vert u \Vert_{L_{t}^2 W_{x}^{1,  \tfrac{2d}{d-2}}([T-l,  T] \times \mathbb{R}^d)}.
\end{aligned}
\label{Hol}
\end{equation}

Secondly,  note that we have established the following Strichartz inequalities:
\begin{equation*}
\Vert e^{it\La} u \Vert_{L_{t}^2 W_{x}^{1,  \tfrac{2d}{d-2}}(\mathbb{R} \times \mathbb{R}^d)} \lesssim \Vert u \Vert_{H_{x}^1(\mathbb{R}^d)},
\end{equation*}

\begin{equation*}
\Vert e^{it\La} u \Vert_{L_{t}^{\infty} H_{x}^{1}(\mathbb{R} \times \mathbb{R}^d)} \lesssim \Vert u \Vert_{H_{x}^1(\mathbb{R}^d)},
\end{equation*}

By Sobolev embedding,  we have
\begin{equation*}
\Vert e^{it\La} u \Vert_{L_{t}^{2} L_{x}^{\tfrac{2d}{d-\gamma}}(\mathbb{R} \times \mathbb{R}^d)} \leqslant
\Vert e^{it\La} u \Vert_{L_{t}^2 W_{x}^{1,  \tfrac{2d}{d-2}}(\mathbb{R} \times \mathbb{R}^d)}
\lesssim \Vert u \Vert_{H_{x}^1(\mathbb{R}^d)},
\end{equation*}

And we deduce
\begin{equation*}
\Vert u \Vert_{L_{t}^2 W_{x}^{1,  \tfrac{2d}{d-2}}([T-l,  T] \times \mathbb{R}^d)} \lesssim \langle l \rangle^{\tfrac{1}{2}},
\end{equation*}

\begin{equation*}
\Vert u \Vert_{L_{t}^{2} L_{x}^{\tfrac{2d}{d-\gamma}}([T-l,  T] \times \mathbb{R}^d)} \lesssim \langle l \rangle^{\tfrac{1}{2}},
\end{equation*}
where $ \langle a \rangle $ denotes $ \left( 1 + \vert a \vert^{2} \right)^{\tfrac{1}{2}} $ for any $ a \in \mathbb{R} $.

Set
\begin{equation*}
X : = \Vert u \Vert_{L_{t}^2 W_{x}^{1,  \tfrac{2d}{d-2}}([T-l,  T] \times \mathbb{R}^d)} + \Vert u \Vert_{L_{t}^{2} L_{x}^{\tfrac{2d}{d-\gamma}}([T-l,  T] \times \mathbb{R}^d)},
\end{equation*}
then by Duhamel formula \eqref{Duh},  we have
\begin{equation*}
\begin{aligned}
X \leqslant & ~ 3 \Vert u(T) \Vert_{H_{x}^1(\mathbb{R}^d)} + 3 \int_{T-l}^{T} \Vert F(u(s)) \Vert_{H_{x}^1(\mathbb{R}^d)} \mathrm{d}x \\
\leqslant & ~ 3 \Vert u(T) \Vert_{H_{x}^1(\mathbb{R}^d)} + \Vert u \Vert_{L_{t}^{\infty}L_{x}^{\tfrac{2d}{d+2-\gamma}}([T-l,  T] \times \mathbb{R}^d)}
\Vert u \Vert_{L_{t}^2 L_{x}^{\tfrac{2d}{d-\gamma}}([T-l,  T] \times \mathbb{R}^d)}
\Vert u \Vert_{L_{t}^2 W_{x}^{1,  \tfrac{2d}{d-2}}([T-l,  T] \times \mathbb{R}^d)} \\
\leqslant & ~ 3 \Vert u(T) \Vert_{H_{x}^1(\mathbb{R}^d)} + \Vert u \Vert_{L_{t}^{\infty}H_x^1([T-l,  T] \times \mathbb{R}^d)}
\Vert u \Vert_{L_{t}^2 L_{x}^{\tfrac{2d}{d-\gamma}}([T-l,  T] \times \mathbb{R}^d)}
\Vert u \Vert_{L_{t}^2 W_{x}^{1,  \tfrac{2d}{d-2}}([T-l,  T] \times \mathbb{R}^d)} \\
\leqslant & ~ 3 \Vert u(T) \Vert_{H_{x}^1(\mathbb{R}^d)} ( 1 + l^{\tfrac{1}{2}} X^{2} ).
\end{aligned}
\end{equation*}

Continuity method yields
\begin{equation*}
X \lesssim \langle l \rangle^{\tfrac{1}{2}} \lesssim \epsilon^{- \tfrac{1}{8}},  ~~~~~~ \text{for} ~ l = \epsilon^{- \tfrac{1}{4}}.
\end{equation*}
Combining with \eqref{eps} and \eqref{Hol},  we have
\begin{equation*}
\int_{T-l}^{T} \Vert f(u(s)) \Vert_{H_{x}^{1} (\mathbb{R}^d)} \mathrm{d}s
\lesssim \epsilon^{\tfrac{1}{2}} \epsilon^{- \tfrac{1}{8}} \epsilon^{- \tfrac{1}{8}}
= \epsilon^{\tfrac{1}{4}},
\end{equation*}
which implies \eqref{nlt} holds.

{\bf Step four. } At the end of the proof,  we prove \eqref{eps} holds.

Note the assumption \eqref{L2dl},  we can see that there exists $ T_1 > 0 $ such that
\begin{equation*}
\int_{\vert x \vert \leqslant R} \vert u(T_1,  x) \vert^2 \mathrm{d}x \leqslant \epsilon^{\tfrac{2}{4-\gamma}}.
\end{equation*}
Fix $ \varphi_{R} \in C_c^{\infty}(\mathbb{R}^d) $,  define the radial function:
\begin{equation*}
\varphi_{R}(x) = \left\{
\begin{aligned}
1,  & ~~~~ \vert x \vert \leqslant \tfrac{R}{2} ; \\
0,  & ~~~~ \vert x \vert \geqslant R,
\end{aligned}
\right.
\end{equation*}

Then we have
\begin{equation*}
\int_{\mathbb{R}^d} \vert \varphi_{R}(x) u(T_1,  x) \vert^2 \mathrm{d}x \leqslant \epsilon^{\tfrac{2}{4-\gamma}}.
\end{equation*}

Using the identity
\begin{equation*}
\partial_{t} \vert u \vert^2 = -2 \nabla \Im ( \overline{u} \nabla u ),
\end{equation*}
which follows from \eqref{problem-eq: NLS-H},  together with integration by parts and Cauchy-Schwarz inequality,  we deduce
\begin{equation*}
\begin{aligned}
 \left\vert \partial_{t} \int_{\mathbb{R}^d} \vert \varphi_{R}(x) \vert^2 \vert u(T_1,  x) \vert^2 \mathrm{d}x \right\vert
\leqslant & ~ 2 \int_{\mathbb{R}^d} \varphi_{R}(x) \vert \nabla \varphi_{R}(x) \vert \vert \Im (\overline{u}\nabla u) \vert \mathrm{d}x \\
\leqslant & ~ \tfrac{6}{R} \Vert u \Vert_{H_{x}^1(\mathbb{R}^d)}^2.
\end{aligned}
\end{equation*}

Thus,  we can find
\begin{equation*}
\begin{aligned}
& \Vert \varphi_{R}(x) u(t,  x) \Vert_{L_{t}^{\infty} L_{x}^2 ([T-l,  T] \times \mathbb{R}^d)} \\
\leqslant & \left\Vert \left( \int_{\mathbb{R}^d} \vert \varphi_{R}(x) \vert^2 \vert u(T_1,  x) \vert^2 \mathrm{d}x + \tfrac{6}{R} \Vert u \Vert_{H_{x}^1(\mathbb{R}^d)}^2 \vert t - T_1 \vert \right)^{\tfrac{1}{2}} \right\Vert_{L_{t}^{\infty} ([T-l,  T])} \\
\leqslant & \left\Vert \left( \epsilon^{\tfrac{2}{4-\gamma}} + \tfrac{6}{R} \Vert u \Vert_{H_{x}^1(\mathbb{R}^d)}^2 \vert t - T_1 \vert \right)^{\tfrac{1}{2}} \right\Vert_{L_{t}^{\infty} ([T-l,  T])}
\lesssim  ~ \epsilon^{\tfrac{1}{4-\gamma}},
\end{aligned}
\end{equation*}
by choosing $ l = \epsilon^{-\tfrac{1}{4}} $ and $ R \geqslant \max\{ \epsilon^{-\tfrac{1}{4}-\tfrac{2}{4-\gamma}},  \epsilon^{-\tfrac{d}{4(\gamma-2)}} \} $,  and $ t,  T_1 \in [T-l,  T]$.

Thus,  by interpolation inequality,  we have
\begin{equation*}
\begin{aligned}
& \Vert \varphi_{R}(x) u(t,  x) \Vert_{L_{t}^{\infty} L_{x}^{\tfrac{2d}{d+2-\gamma}} ([T-l,  T] \times \mathbb{R}^d)} \\
\leqslant & \Vert \varphi_{R}(x) u(t,  x) \Vert_{L_{t}^{\infty} L_{x}^2 ([T-l,  T] \times \mathbb{R}^d)}^{\tfrac{4-\gamma}{2}} \Vert \varphi_{R}(x) u(t,  x) \Vert_{L_{t}^{\infty} L_{x}^{\tfrac{2d}{d-2}} ([T-l,  T] \times \mathbb{R}^d)}^{\tfrac{\gamma-2}{2}} \\
\leqslant &\Vert \varphi_{R}(x) u(t,  x) \Vert_{L_{t}^{\infty} L_{x}^2 ([T-l,  T] \times \mathbb{R}^d)}^{\tfrac{4-\gamma}{2}} \Vert \varphi_{R}(x) u(t,  x) \Vert_{L_{t}^{\infty} H_{x}^{1} (\mathbb{R} \times \mathbb{R}^d)}^{\tfrac{\gamma-2}{2}}
\lesssim \epsilon^{\tfrac{1}{2}}
\end{aligned}
\end{equation*}
and
\begin{equation*}
\begin{aligned}
& \Vert ( 1 - \varphi_{R}(x) ) u(t,  x) \Vert_{L_{t}^{\infty} L_{x}^{\tfrac{2d}{d+2-\gamma}} ([T-l,  T] \times \mathbb{R}^d)} \\
\leqslant & \Vert ( 1 - \varphi_{R}(x) ) u(t,  x) \Vert_{L_{t}^{\infty} L_{x}^{\infty} ([T-l,  T] \times \mathbb{R}^d)}^{\tfrac{\gamma-2}{d}} \Vert ( 1 - \varphi_{R}(x) ) u(t,  x) \Vert_{L_{t}^{\infty} L_{x}^{2} ([T-l,  T] \times \mathbb{R}^d)}^{\tfrac{d+2-\gamma}{d}} \\
\leqslant & \left(\tfrac{4}{R^2}\right)^{\tfrac{\gamma-2}{d}} \Vert u(t,  x) \Vert_{L_{t}^{\infty} H_{x}^{1} (\mathbb{R} \times \mathbb{R}^d)}^{\tfrac{d+2-\gamma}{d}}
\lesssim \epsilon^{\tfrac{1}{2}}.
\end{aligned}
\end{equation*}
So,
\begin{equation}
\Vert u(t,  x) \Vert_{L_{t}^{\infty}L_{x}^{\tfrac{2d}{d+2-\gamma}}([T-l,  T] \times \mathbb{R}^d)}
\lesssim \epsilon^{\tfrac{1}{2}}
\end{equation}
for $ l = \epsilon^{-\tfrac{1}{4}} $.

Combing the above steps,  we finish the proof of the scattering criterion only if $ \epsilon > 0 $ is arbitrarily small.
\end{proof}

\newpage
{\center\section{Morawetz Estimate}}
In this section,  we are now in the position to prove the  Morawetz Estimate \eqref{Me} holds.
 As we all know,  the decay estimate of the solution $ u $ can be  characterized by Morawetz estimate.

We define the function:
\begin{equation}
M(t)=2\Im\int_{\mathbb{R}^d}\overline{u}\nabla u \cdot \nabla \omega \mathrm{d}x,
\label{Mar}
\end{equation}
where $ u = u(t,  x) $ is the solution of \eqref{problem-eq: NLS-H},  and $ \omega = \omega(x) $ is a real function to be chosen later.

Then we have
\begin{equation}
\begin{aligned}
\tfrac{\mathrm{d}}{\mathrm{d}t}M(t)
= & ~ 4\sum_{k=1}^d\sum_{j=1}^d\Re\int_{\mathbb{R}^d}\overline{u_{k}} u_j \omega_{kj} \mathrm{d}x
+ \int_{\mathbb{R}^d}\vert u \vert^2 (- \Delta \Delta \omega ) +4|u|^2\frac{ax}{|x|^4}\cdot\nabla \omega\mathrm{d}x \\
& + (-\gamma)\int_{\mathbb{R}^d}\int_{\mathbb{R}^d}\tfrac{ x-y }{\vert x-y \vert^{\gamma+2}} \vert u(x) \vert^{2} \vert u(y) \vert^2 \cdot (\nabla \omega(x) - \nabla \omega(y)) \mathrm{d}x \mathrm{d}y,
\end{aligned}
\label{MeH1}
\end{equation}
where we have used
\begin{equation*}
\begin{aligned}
& \int_{\mathbb{R}^d}(\tfrac{ \cdot }{\vert \cdot \vert^{\gamma+2}} * \vert u \vert^{2})\vert u \vert^2 \cdot \nabla \omega \mathrm{d}x \\
= & \int_{\mathbb{R}^d}\int_{\mathbb{R}^d}\tfrac{ x-y }{\vert x-y \vert^{\gamma+2}} \vert u(y) \vert^{2} \vert u(x) \vert^2 \cdot \nabla \omega(x) \mathrm{d}y \mathrm{d}x \\
= & \int_{\mathbb{R}^d}\int_{\mathbb{R}^d}\tfrac{ y-x }{\vert y-x \vert^{\gamma+2}} \vert u(x) \vert^{2} \vert u(y) \vert^2 \cdot \nabla \omega(y) \mathrm{d}x \mathrm{d}y \\
= & ~ \tfrac{1}{2}\int_{\mathbb{R}^d}\int_{\mathbb{R}^d}\tfrac{ x-y }{\vert x-y \vert^{\gamma+2}} \vert u(x) \vert^{2} \vert u(y) \vert^2 \cdot (\nabla \omega(x) - \nabla \omega(y)) \mathrm{d}x \mathrm{d}y.
\end{aligned}
\end{equation*}
For fixed $ R \gg 1 $,  we choose $ \omega(x) $ in \eqref{Mar} to be a hybrid function
\begin{equation}
\omega(x)=
\left\{
\begin{aligned}
\vert x \vert^2,  ~~~~ & \vert x \vert \leqslant R,  \\
3R \vert x \vert,  ~~~~ & \vert x \vert \geqslant 2R
\end{aligned}
\right.
\label{ax}
\end{equation}
satisfying
\begin{equation*}
\begin{aligned}
(\romannumeral1) ~ & \partial_{r} \omega \geqslant 0 ~ ; \\
(\romannumeral2) ~ & \partial_{r}^2 \omega \geqslant 0 ~ ; \\
(\romannumeral3) ~ & \vert \partial_{\alpha} \omega(x) \vert
\leqslant C_{\alpha} R \vert x \vert^{-\vert \alpha \vert +1},  ~~~~
\text{when}~ R < \vert x \vert < 2R ~.
\end{aligned}
\end{equation*}
Here $ \partial_r $ denotes the radial derivative,  i. e.
$ \partial_r \omega = \nabla \omega \cdot \tfrac{x}{\vert x \vert} $.  Under these conditions,  the matrix $ \omega_{jk} $ is non-negative.  And
\begin{align*}
\end{align*}

\begin{equation}
\begin{aligned}
\tfrac{\mathrm{d}}{\mathrm{d}t} M(t) \geqslant & ~ \left( 8 \sum_{k=1}^d \sum_{j=1}^d \Re \int_{\vert x \vert \leqslant R} \overline{u_{k}} u_j \delta_{kj} \mathrm{d}x
- 2\gamma \int_{\vert x \vert \leqslant R}\int_{\mathbb{R}^d}\tfrac{ \vert u(x) \vert^{2} \vert u(y) \vert^2 }{\vert x-y \vert^{\gamma}} \mathrm{d}y \mathrm{d}x \right) \\
& + 12 \sum_{k=1}^d \sum_{j=1}^d \Re \int_{\vert x \vert \geqslant 2R} \tfrac{R}{\vert x \vert}
\left[ \delta_{jk} - \tfrac{x_j}{\vert x \vert} \tfrac{x_k}{\vert x \vert} \right] \overline{u_{k}} u_j \mathrm{d}x
+ 24 \int_{\vert x \vert \geqslant 2R} \tfrac{\vert u \vert^2 R}{\vert x \vert^3} \mathrm{d}x \\
& - 3\gamma \int_{\vert x \vert \geqslant 2R} \int_{\mathbb{R}^d}
\tfrac{ x-y }{\vert x-y \vert^{\gamma+2}} \vert u(x) \vert^{2} \vert u(y) \vert^{2}
\cdot \left( \tfrac{Rx}{\vert x \vert} - \tfrac{Ry}{\vert y \vert} \right) \mathrm{d}y \mathrm{d}x \\
& + \int_{ R < \vert x \vert < 2R } \mathcal{O}( \tfrac{\vert u \vert^2 R}{\vert x \vert^3} ) \mathrm{d}x + \int_{ R < \vert x \vert < 2R } \int_{\mathbb{R}^d} \mathcal{O}( \tfrac{ \vert u(x) \vert^{2} \vert u(y) \vert^2 }{\vert x-y \vert^{\gamma}} ) \mathrm{d}y \mathrm{d}x \\
+&8\int_{\vert x \vert \leqslant R} a\frac{|u|^2}{|x|^2}\mathrm{d}x+
\int_{ R < \vert x \vert < 2R } 4|u|^2\frac{ax}{|x|^4}\cdot\nabla \omega\mathrm{d}x+
 12R\int_{\vert x \vert \geqslant 2R}\frac{a|u|^2}{|x|^3}\mathrm{d}x\\
= & : \mathcal{A + B + C + D + E + F}.
\end{aligned}
\label{MeHa}
\end{equation}

We will make use of radial Sobolev embedding to estimate every terms in \eqref{MeHa}.  Now we  choose a cut function $ \chi_R \in C^{\infty}(\mathbb{R}^d) $,  satisfying
\begin{equation*}
\chi_R (x) = \left\{
\begin{aligned}
1,  ~~~~ & \vert x \vert \leqslant \tfrac{R}{2},  \\
0,  ~~~~ & \vert x \vert \geqslant R.
\end{aligned}
\right.
\end{equation*}

For $ \mathcal{A} $,  according to Lemma \ref{coerb},
\begin{equation*}
\begin{aligned}
\mathcal{A}
= & ~ 8 \int_{\vert x \vert \leqslant R} \vert \nabla u \vert^2
\mathrm{d}x +8\int_{\vert x \vert \leqslant R} a\frac{|u|^2}{|x|^2}\mathrm{d}x- 2\gamma \int_{\vert x \vert \leqslant \tfrac{R}{2}}
\int_{\vert y \vert \leqslant \tfrac{R}{2}}
\tfrac{ \vert u(x) \vert^{2} \vert u(y) \vert^2 }
{\vert x-y \vert^{\gamma}} \mathrm{d}y \mathrm{d}x \\
& - 2\gamma \int_{\tfrac{R}{2} < \vert x \vert \leqslant R}
\int_{\vert y \vert \leqslant \tfrac{R}{2}}
\tfrac{ \vert u(x) \vert^{2} \vert u(y) \vert^2 }
{\vert x-y \vert^{\gamma}} \mathrm{d}y \mathrm{d}x
-2\gamma \int_{\vert x \vert \leqslant R}\int_{\vert y \vert > \tfrac{R}{2}}
\tfrac{ \vert u(x) \vert^{2} \vert u(y) \vert^2 }
{\vert x-y \vert^{\gamma}} \mathrm{d}y \mathrm{d}x \\
\geqslant & ~ 8 \left[ \Vert \chi_{R} u \Vert_{\dot{H}_{a}^1(\mathbb{R}^d)}^2
- \tfrac{\gamma}{4} P(\chi_{R} u) \right] \\
& - 2\gamma\int_{\tfrac{R}{2} < \vert x \vert \leqslant R}
\int_{\vert y \vert \leqslant \tfrac{R}{2}}
\tfrac{ \vert u(x) \vert^{2} \vert u(y) \vert^2 }
{\vert x-y \vert^{\gamma}} \mathrm{d}y \mathrm{d}x
- 2\gamma\int_{\vert x \vert \leqslant R}\int_{\vert y \vert > \tfrac{R}{2}}
\tfrac{ \vert u(x) \vert^{2} \vert u(y) \vert^2 }
{\vert x-y \vert^{\gamma}} \mathrm{d}y \mathrm{d}x \\
\gtrsim & ~ \delta^\prime P(\chi_{R} u) - \mathcal{A}^\prime - \mathcal{A}^{\prime\prime},
\end{aligned}
\end{equation*}
where
\begin{equation*}
\begin{aligned}
\mathcal{A}^\prime = & ~ 2\gamma \int_{\tfrac{R}{2} < \vert x \vert \leqslant R}
\int_{\vert y \vert \leqslant \tfrac{R}{2}}
\tfrac{ \vert u(x) \vert^{2} \vert u(y) \vert^2 }
{\vert x-y \vert^{\gamma}} \mathrm{d}y \mathrm{d}x \\
\lesssim  & ~ \tfrac{\gamma}{R^\alpha} \int_{\tfrac{R}{2} < \vert x \vert \leqslant R}
\int_{\vert y \vert \leqslant \tfrac{R}{2}}
\vert x \vert^{2\alpha} \vert u(x) \vert^{\alpha}
\tfrac{ \vert u(x) \vert^{2-\alpha} \vert u(y) \vert^2 }
{\vert x-y \vert^{\gamma}} \mathrm{d}y \mathrm{d}x \\
\lesssim & ~  \tfrac{\gamma}{R^\alpha}  \Vert u \Vert_{L_{x}^2(\mathbb{R}^d)}^{\tfrac{\alpha}{2}}
\Vert \nabla u \Vert_{L_{x}^2(\mathbb{R}^d)}^{\tfrac{\alpha}{2}}
\int_{\tfrac{R}{2} < \vert x \vert \leqslant R}
\int_{\vert y \vert \leqslant \tfrac{R}{2}}
\tfrac{ \vert u(x) \vert^{2-\alpha} \vert u(y) \vert^2 }
{\vert x-y \vert^{\gamma}} \mathrm{d}y \mathrm{d}x \\
\lesssim & ~ \tfrac{\gamma}{R^\alpha}  \Vert u \Vert_{L_{x}^2(\mathbb{R}^d)}^{2-\tfrac{\alpha}{2}}
\Vert \nabla u \Vert_{L_{x}^2(\mathbb{R}d)}^{\tfrac{\alpha}{2}}
\left\Vert \int_{\mathbb{R}^d} \tfrac{ \vert u(y) \vert^{2}}{\vert x-y \vert^{\gamma}}
\mathrm{d}y \right\Vert_{L_{x}^q(\mathbb{R}^d)} \\
\lesssim & ~  \tfrac{\gamma}{R^\alpha}  \Vert u \Vert_{L_{x}^2(\mathbb{R}^d)}^{2-\tfrac{\alpha}{2}}
\Vert \nabla u \Vert_{L_{x}^2(\mathbb{R}d)}^{\tfrac{\alpha}{2}}
\Vert u \Vert_{L_{x}^{p}(\mathbb{R}^d)}^2 \\
\lesssim & ~  \tfrac{\gamma}{R^\alpha}  \Vert u \Vert_{L_{x}^2(\mathbb{R}^d)}^{3}
\Vert \nabla u \Vert_{L_{x}^2(\mathbb{R}^d)},
\end{aligned}
\end{equation*}
where we need
\begin{align*}
1=\tfrac{2-\alpha}{2}+\tfrac1{q},\\
\tfrac1{q}+1=\tfrac{\gamma}d+\tfrac2{p},\\
2<p<\frac{2d}{d-2},
\end{align*}
which yields
$$\alpha=\tfrac{2\gamma}{d}+\tfrac4{p}-1,\\
\tfrac{2d}{2d-\gamma}\leq p<\tfrac{2d}{d-2}$$
and by the same argument,
\begin{equation*}
\begin{aligned}
\mathcal{A}^{\prime\prime}
\leqslant & ~ \tfrac{C}{R^2} \Vert u \Vert_{L_{x}^2(\mathbb{R}^d)}^{3}
\Vert \nabla u \Vert_{L_{x}^2(\mathbb{R}^d)}.
\end{aligned}
\end{equation*}

For $ \mathcal{B} $,  because of
$ \not\nabla u = \nabla u - \tfrac{x}{\vert x \vert}\partial_{r}u $,
and $ \partial_{r}u = \tfrac{x}{\vert x \vert}\nabla u $,  we have
\begin{equation*}
\not\nabla u = \nabla u - \tfrac{x}{\vert x \vert} \tfrac{x \cdot \nabla u}{\vert x \vert}.
\end{equation*}
so,
\begin{equation*}
\begin{aligned}
\vert\not\nabla u \vert^2
= & ~ \vert \nabla u \vert^2 - \sum_{j=1}^d \sum_{k=1}^d \tfrac{x_j}{\vert x \vert} \tfrac{x_k}{\vert x \vert} u_j \overline{u_k}.
\end{aligned}
\end{equation*}

For $ u $ is radial,  thus
\begin{equation*}
\begin{aligned}
\mathcal{B}
= & ~ 4 \sum_{k=1}^d \sum_{j=1}^d \Re \int_{\vert x \vert \geqslant 2R} \tfrac{R}{\vert x \vert}
\left[ \delta_{jk} - \tfrac{x_j}{\vert x \vert} \tfrac{x_k}{\vert x \vert} \right] \overline{u_{k}} u_j \mathrm{d}x \\
= & ~ 4 \Re \int_{\vert x \vert \geqslant 2R} \tfrac{R}{\vert x \vert}
\left[ \vert \nabla u \vert^2 - \sum_{j=1}^d \sum_{k=1}^d \tfrac{x_j}{\vert x \vert} \tfrac{x_k}{\vert x \vert} u_j \overline{u_k} \right] \mathrm{d}x \\
= & ~ 4 \int_{\vert x \vert \geqslant 2R} \tfrac{R}{\vert x \vert} ~ \vert\not\nabla u \vert^2 \mathrm{d}x
= ~ 0.
\end{aligned}
\end{equation*}

It is obvious that
\begin{equation*}
\mathcal{C} \geqslant 0,  ~~~~
\vert \mathcal{E} \vert \lesssim \tfrac{1}{R^2} \Vert u \Vert_{L_{x}^2(\mathbb{R}^d)}.
\end{equation*}
Similar to the estimates of $ \mathcal{A}^{\prime} $ and $ \mathcal{A}^{\prime\prime} $,  we have
\begin{equation*}
\vert  \mathcal{D} \vert
\leqslant \tfrac{C}{R} \Vert u \Vert_{L_{x}^2(\mathbb{R}^d)}^{2}
\Vert \nabla u \Vert_{L_{x}^2(\mathbb{R}d)}^{2}
\end{equation*}
and
\begin{equation*}
\vert \mathcal{F} \vert
\leqslant \tfrac{C}{R^2} \Vert u \Vert_{L_{x}^2(\mathbb{R}^d)}^{3}
\Vert \nabla u \Vert_{L_{x}^2(\mathbb{R}^d)}.
\end{equation*}

Continuing from above and \eqref{MeHa},  we discard non-negative terms and deduce
\begin{equation*}
\begin{aligned}
\tfrac{\mathrm{d}}{\mathrm{d}t} M(t) = & ~ \mathcal{A + B + C + D + E + F + G} \\
\gtrsim & ~ \delta^\prime P(\chi_{R} u) - \mathcal{A}^\prime - \mathcal{A}^{\prime\prime}
+ \mathcal{D + E + F},
\end{aligned}
\end{equation*}
which implies
\begin{equation*}
\begin{aligned}
\delta^\prime P(\chi_{R} u) \lesssim & ~ \tfrac{\mathrm{d}}{\mathrm{d}t} M(t)
+ \mathcal{A}^\prime + \mathcal{A}^{\prime\prime} + \vert \mathcal{D} \vert
+ \vert \mathcal{E} \vert + \vert \mathcal{F} \vert \\
\lesssim & ~ \tfrac{\mathrm{d}}{\mathrm{d}t} M(t) + \tfrac{1}{R^2} + \tfrac{1}{R} \\
\lesssim & ~ \tfrac{\mathrm{d}}{\mathrm{d}t} M(t) + \tfrac{1}{R}.
\end{aligned}
\end{equation*}

The fundamental theorem of calculus tells us
\begin{equation*}
\delta^\prime \int_0^T P(\chi_{R} u) \mathrm{d}t \lesssim \sup_{t \in [0,  T]} \vert M(t) \vert
+ \tfrac{T}{R},
\end{equation*}
or expressed as
\begin{equation*}
\tfrac{1}{T}\int_0^T P(\chi_{R} u) \mathrm{d}t \lesssim_{\delta,  u} \tfrac{1}{T} + \tfrac{1}{R}.
\end{equation*}
Thus,  we have proved \eqref{Me} exactly and finished the proof.

\end{document}